%% file: hooks.tex
\documentclass{article}
\usepackage{fullpage}
\usepackage[english]{babel}
\usepackage[cmex10]{amsmath}
\usepackage[all=normal,paragraphs=tight,bibliography=tight]{savetrees}
\usepackage{todonotes}
\usepackage{amsthm}
\usepackage{amssymb}
\usepackage{microtype}
\usepackage{xspace}
\usepackage{todonotes}
\usepackage{comment}
\usepackage{enumerate}
\usepackage{tikz}
\usepackage{graphicx}
\usetikzlibrary{arrows, decorations.markings}
\usetikzlibrary{fit}					
\usetikzlibrary{backgrounds}

\usepackage{graphicx}
\usepackage{caption}
\usepackage{subcaption}

\numberwithin{equation}{section}
\newtheorem{theorem}{Theorem}[section]
\newtheorem{lemma}[theorem]{Lemma}
\newtheorem{claim}[theorem]{Claim}
\newtheorem{corollary}[theorem]{Corollary}

\newtheorem{conjecture}[theorem]{Conjecture}
\theoremstyle{definition}
\newtheorem{definition}[theorem]{Definition}

\def\cqedsymbol{\ifmmode$\lrcorner$\else{\unskip\nobreak\hfil
\penalty50\hskip1em\null\nobreak\hfil$\lrcorner$
\parfillskip=0pt\finalhyphendemerits=0\endgraf}\fi}

\newcommand{\hook}[1]{\ensuremath{{H}_{#1}}}
\newcommand{\dhook}[1]{\hook{#1}^2}
\newcommand{\dhookfam}[1]{\ensuremath{\cH^2_{\geq #1}}}

\renewcommand{\ge}{\geqslant} 

\newcommand{\se}{\subseteq}
\newcommand{\sm}{\setminus}

\newcommand{\Exp}{\mathbf{E}}
\newcommand{\Prob}{\mathbf{Pr}}

\newcommand{\bheps}{\delta}
\newcommand{\eps}{\varepsilon}
\newcommand{\degeps}{\varepsilon}
\newcommand{\acteps}{\gamma}

\newcommand{\cF}{\ensuremath{\mathcal{F}}}
\newcommand{\cG}{\ensuremath{\mathcal{G}}}
\newcommand{\cH}{\ensuremath{\mathcal{H}}}

\newcommand{\cS}{\mathcal{S}}

\definecolor{myred}{RGB}{220,24,10}

\newcommand{\EH}{Erd\H{o}s-Hajnal }
\newcommand{\compl}{{\mathrm{c}}}

\newcommand{\ABsym}{\ensuremath{A \leftrightarrow B}}

\newcommand{\reach}[2]{\mathtt{reach}(#1 \to #2)}

\tikzstyle{every node}=[circle, draw, fill=black!50, inner sep=0pt, minimum width=20pt]

\tikzstyle{input}=[circle,
                                    thick,
                                    minimum size=2.2cm,
                                    draw=black!80,
                                    fill=white!20]

\tikzstyle{input2}=[circle,
                                    thick,
                                    minimum size=1.0cm,
                                    draw=black!80,
                                    fill=white!20]

\tikzstyle{matrx}=[rectangle,
                                    thick,
                                    minimum size=1.8cm,
                                    draw=black!80,
                                    fill=white!20]

\tikzstyle{matrx2}=[rectangle,
                                    thick,
                                    minimum size=1.5cm,
                                    draw=black!80,
                                    fill=gray!20]

\tikzstyle{vecArrow} = [thick, decoration={markings,mark=at position
   1 with {\arrow[semithick]{open triangle 60}}},
   double distance=1.8pt, shorten >= 5.5pt,
   preaction = {decorate},
   postaction = {draw,line width=1.4pt, white,shorten >= 4.5pt}]
\tikzstyle{innerWhite} = [semithick, white,line width=1.4pt, shorten >= 4.5pt]

\tikzstyle{background}=[rectangle,
                                                fill=gray!10,
                                                inner sep=0.2cm,
                                                rounded corners=5mm]

\title{Excluding hooks and their complements}
\author{%
Krzysztof Choromanski\thanks{Google Research New York, 76 Ninth Ave, New York, NY 10011, USA, kchoro@google.com.}
\and
Dvir Falik\thanks{School of Mathematical Sciences,
Queen Mary - University of London,
Mile End Road,
London, E1 4NS, UK, \texttt{dvir.falik@gmail.com}. Research supported by the Warwick-QMUL Alliance in Advances in Discrete Mathematics and its Applications.}
\and
Anita Liebenau\thanks{Department of Computer Science and DIMAP, University of Warwick, Coventry CV4 7AL, UK, \texttt{a.liebenau@warwick.ac.uk}. Research supported by the European Research Council under the European Union�s Seventh
Framework Programme (FP7/2007- 2013)/ERC grant agreement no.~259385.}
\and
Viresh Patel\thanks{Korteweg-de Vries Institute for Mathematics, University of Amsterdam, Science Park 904,
1098 XH,  Amsterdam, 
The Netherlands, \texttt{vpatel@uva.nl}. Research partially supported by the Warwick-QMUL Alliance in Advances in Discrete Mathematics and its Applications and by the Netherlands Organisation for Scientific Research (NWO) through the Gravitation Programme Networks (024.002.003).}
\and
Marcin Pilipczuk\thanks{Department of Computer Science and DIMAP, University of Warwick, Coventry CV4 7AL, UK, \texttt{m.pilipczuk@warwick.ac.uk}. Research partially supported by the Centre for Discrete Mathematics and its Applications (DIMAP) at the University of Warwick and by the Warwick-QMUL Alliance in Advances in Discrete Mathematics and its Applications.}}

\date{}

\begin{document}

\maketitle

\begin{abstract}
The celebrated \EH conjecture states that for every $n$-vertex undirected graph $H$ there exists $\eps(H)>0$ such that every graph $G$ that does not contain $H$ as an induced subgraph contains a clique or an independent set of size at least $n^{\eps(H)}$.
A weaker version of the conjecture states that the polynomial-size clique/independent set phenomenon occurs if one excludes both $H$ and its complement $H^{\compl}$.
We show that the weaker conjecture holds if $H$ is any path with a pendant edge at its third vertex; thus we give a new infinite family of graphs for which the conjecture holds. 
\end{abstract}

\section{Introduction}
\input{introduction}

\section{Preliminaries}\label{sec:prelims}
\input{prelims}

\section{Line graphs and claw-free Berge graphs}\label{sec:claw-free}
\input{claw-free}

\section{Double hooks have the strong \EH property}\label{sec:big}
\input{big-picture}

\section{Warm up: Proof of Theorem~\ref{thm:tech-claw-free}}\label{sec:warm-up}
\input{tech-claw-free}

\section{Proof of Theorem~\ref{thm:tech-hooks}}\label{sec:structuralTheorem}
\input{tech-hooks}

\section{The strong \EH property is much stronger} 
\label{sec:negative}
\input{bihomo-discussion}

\section{Conclusions}\label{sec:concl}
\input{conclusions}

\bibliographystyle{abbrv}
\bibliography{erdos-hajnal}

\end{document}

%% file: introduction.tex
The Erd\H{o}s-Hajnal conjecture is a long-standing and intensely-studied conjecture in Ramsey Theory bringing together extremal, structural and probabilistic aspects of graph theory. Informally it says that if a large graph $G$ does not contain a fixed graph $H$ as an induced subgraph then $G$ contains a large clique or independent set.

More formally, a set of vertices in a graph $G$  is {\em homogeneous} if it induces a clique or independent set, and we denote the largest homogeneous set of $G$ by $\hom(G)$. Given a graph $H$ (resp.\ a family of graphs $\cF=\{H_1, H_2, \ldots\}$), a graph $G$ is said to be $H$-free (resp.\ $\cF$-free) if $G$ does not contain $H$ (resp.\ any member of $\cF$) as an induced subgraph. The famous Erd\H{o}s-Hajnal conjecture \cite{eh1977} is the following.

\begin{conjecture}\label{conj:EH}
For every graph $H$, there exists $c(H) > 0$ such that if $G$ is an $n$-vertex $H$-free graph then $\hom(G)\ge n^{c(H)}$. 
\end{conjecture}  

The Erd\H{o}s-Szekeres bounds on Ramsey numbers \cite{es1935} imply that $\hom(G) > \frac{1}{2}\log_2 n$ for any $n$-vertex graph $G$. Furthermore, the lower bounds on Ramsey numbers found by Erd\H{o}s \cite{e1947} using the probabilistic method imply that for a uniformly random $n$-vertex graph, $\hom(G) < O(\log n)$ with high probability. Thus the Erd\H{o}s-Hajnal conjecture suggests that $H$-free graphs are quite different from typical (random) graphs with respect to homogeneous sets. The best known bound for the above conjecture is due to 
Erd\H{o}s and Hajnal~\cite{eh1989}: they showed that the conjecture above holds if we replace $n^{c(H)}$ with $e^{c(H)\sqrt{\ln n}}$.

Despite much attention, the conjecture is known to hold for only a limited choice of $H$. We give some brief background here and refer the interested reader to the excellent survey \cite{c2014} by Chudnovsky on this fascinating conjecture.   
The upper bounds on Ramsey numbers due to Erd\H{o}s and Szekeres \cite{es1935} 
immediately imply that $K_k$, the clique on $k$ vertices, satisfies the \EH conjecture. 
%
In \cite{aps2001}, Alon, Pach and Solymosi prove the \EH conjecture for a (non-trivial)   
infinite family of graphs by showing that if the conjecture is true 
for two graphs $H_1$ and $H_2$, then it is also true for the graph formed by blowing up a vertex of $H_1$ and inducing a copy of $H_2$ amongst the new vertices.
It is known that the \EH conjecture holds for all graphs on at most four vertices and using the result in \cite{aps2001} immediately implies that the \EH conjecture holds for all graphs on five vertices except the four-edge path, its complement, the 
cycle on five vertices, and a graph commonly called the {\em bull} 
(a triangle with two pendant edges). 
Chudnovsky and Safra \cite{cs2008} settled the conjecture for the bull, but the question remains unsolved for the other three graphs on five vertices. 


Rather than considering $H$-free graphs, 
one can weaken the \EH conjecture by considering hereditary graph classes (i.e.\ graph classes that are closed under taking induced subgraphs). In this direction the following conjecture was proposed, see e.g.~\cite{c2014}, \cite{g1997}. We denote by $H^{\compl}$ the complement of a graph $H$. 
\begin{conjecture}\label{conj:EHLight}
For every graph $H$, there exists a constant $c(H)> 0$ such that if $G$ is an $n$-vertex $\{H,H^{\compl}\}$-free 
graph then $\hom(G) \ge n^{c(H)}$. 
\end{conjecture}
Compared to the \EH conjecture, this conjecture is known to hold for only a few additional choices of $H$. Let $P_k$ be the path on $k$ vertices. Chudnovsky and Seymour \cite{cs2012} proved Conjecture~\ref{conj:EHLight} when $H$ is  $P_6$. Recently, Bousquet, Lagoutte and Thomass\'e \cite{blt2015} generalised this result and showed that, for every $k$, the conjecture above holds when $H$ is taken to be $P_k$. Our main contribution is to generalise this result from $k$-paths to what we call $k$-hooks, giving a new infinite family of $H$ for which Conjecture~\ref{conj:EHLight} holds.

A \emph{$k$-hook}, denoted by $\hook{k}$, is the graph on $k+4$ vertices  
$\{v_1,\ldots,v_{k+4}\}$, where $\{v_1,\ldots,v_{k+3}\}$ form a $(k+3)$-vertex path, and $v_{k+1}v_{k+4}$ 
is a pendant edge. 
An illustration can be found in Figure~\ref{fig:Picture1}. 
We call this graph a $k$-hook (rather than, say, a $(k+3)$-hook), 
since it is more convenient to treat it like a $k$-vertex path with a {\em hook}, i.e.~a four-vertex path, attached to it. 
\begin{theorem}\label{thm:EH-hooks}
For every $k\geq 1$ there exists $c_k$ such that if $G$ is an $n$-vertex $\{\hook{k},\hook{k}^\compl\}$-free graph then $\hom(G)>n^{c_k}$. 
\end{theorem}

 \begin{figure}[b]
 \centering
\phantom{asdfasdfasdfaasdfasfasdfad}
\begin{subfigure}[b]{0.4\textwidth}
	\centering
	 \includegraphics[width=0.8\textwidth]
	 {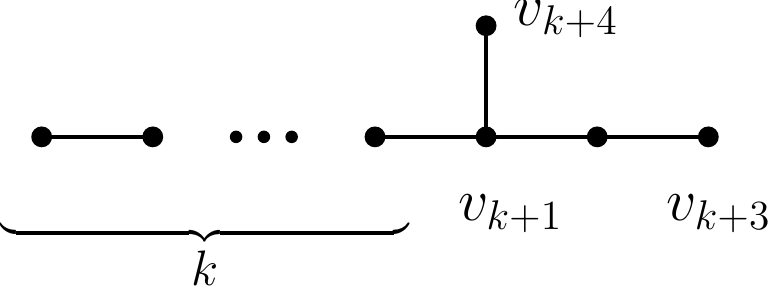}
	 \caption{A $k$-hook.}
	 \label{fig:Picture1}
\end{subfigure} 
\qquad
 \begin{subfigure}[b]{0.45\textwidth}
 	\centering
	\includegraphics[width=0.9\textwidth]
	{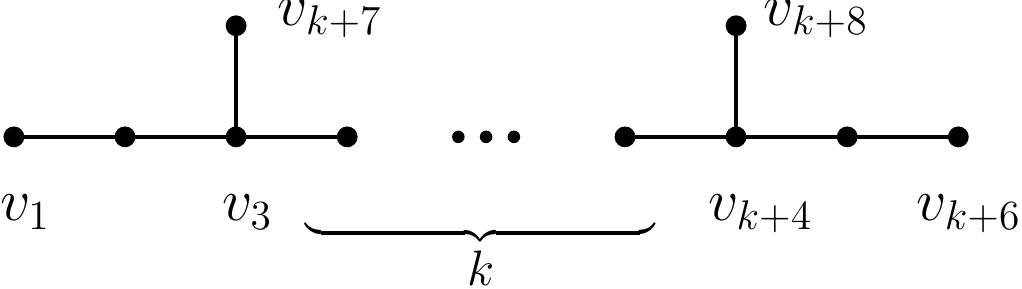}
	 \caption{A double $k$-hook.} 
	 \label{fig:Picture2}
 \end{subfigure}
 \caption{}
 \label{fig:Pictures1and2}
 \end{figure}

We now describe some further results.
A class $\cG$ of graphs is said to have the {\em (weak) \EH property} if there exists a constant $c>0$ such that every graph $G\in \cG$ satisfies $\hom(G)\ge n^c$ where $n$ is the number of vertices in $G$. Clearly, Conjecture~\ref{conj:EH} is equivalent to the statement  that, for every 
graph $H$, the class of $H$-free graphs satisfies the \EH property. 
Instead of asking for homogeneous sets, one can ask for homogeneous pairs: for a graph $G$ and two disjoint subsets of its vertices $P$ and $Q$, we say that $(P,Q)$ is a {\em homogeneous pair} if every edge between $P$ and $Q$ is present or if every edge between $P$ and $Q$ is absent. In the former case, we call $(P,Q)$ an {\em adjacent pair} and in the latter case an {\em anti-adjacent pair}. A graph class $\mathcal{G}$ has \emph{the strong Erd\H{o}s-Hajnal property} if there exists a constant $\bheps > 0$
such that every $G \in \mathcal{G}$ with at least two vertices has a homogeneous pair $(P,Q)$ with $|P|,|Q| \geq \bheps |V(G)|$. It is not hard to show (see, e.g., \cite{apprs2005}, \cite{fp2008}) that if a graph class $\cG$ has the strong \EH property, 
then it also has the (weak) \EH property.

We shall prove Theorem~\ref{thm:EH-hooks} by proving that a more general graph class has the strong \EH property.
A {\em double $k$-hook}, denoted by $\dhook{k}$, is the graph 
on $k+8$ vertices $\{v_1,\ldots,v_{k+8}\}$, where the vertices $\{v_1,\ldots,v_{k+6}\}$ 
form a $(k+6)$-vertex path, and $v_3v_{k+7}$ and $v_{k+4},v_{k+8}$ are 
pendant edges.  
Again, we prefer to view this graph as a $k$-vertex path, with a hook attached to each end of the path: 
hence the name. 
For an illustration, see Figure~\ref{fig:Picture2}. 
Let $\dhookfam{k}:= \{\dhook{\ell}, (\dhook{\ell})^{\compl}: \ell \geq k\}$
i.e.\ the set of double $\ell$-hooks and their complements for all $\ell \geq k$. 
Since the class of $\{\hook{k},\hook{k}^\compl\}$-free graphs is a subclass of $\dhookfam{k}$-free graphs, 
Theorem \ref{thm:EH-hooks} is implied by the following. 

\begin{theorem}\label{thm:main-hook}
For every $k \geq 1$, the class of $\dhookfam{k}$-free
graphs has the strong Erd\H{o}s-Hajnal property.
\end{theorem}

Note that the result in \cite{blt2015} mentioned earlier, that the class of $\{P_k,P_k^\compl\}$-free graphs has \EH property, is in fact proved by showing that $\{P_k,P_k^\compl\}$-free graphs have the strong \EH property. 
Furthermore, Bonamy, Bousquet and Thomass\'e \cite{bbt2014} show that $\cG_k$ has the strong \EH property, where 
$\cG_k$ is the class of graphs that do not contain the cycle $C_{\ell}$ on $\ell$ vertices or its complement $C_{\ell}^{\compl}$ as an induced subgraph for all $\ell \geq k$.


%
%
%
%
%
%
%
%

In the course of the paper, we shall prove that two further hereditary graph classes have the strong \EH property. We believe these results may be of independent interest. A \emph{hole} in a graph is an induced cycle of length at least $4$ and an {\em antihole} is the complement of such a graph.
A {\em Berge graph} is a graph that does not contain any odd hole or odd antihole.
It follows easily from the Strong Perfect Graph Theorem that the class of Berge graphs satisfies the (weak) \EH property, but a certain random poset construction  \cite{f2006} shows that it does not satisfy the strong \EH property. However if we also forbid the claw, i.e.\ the star on four vertices then the strong \EH property holds. 
\begin{theorem}\label{thm:main-claw-free-intro}
The class of claw-free Berge graphs has the strong Erd\H{o}s-Hajnal property.
\end{theorem}
\noindent
In \cite{lt2015}, Lagoutte and Trunck show that another subclass of Berge graphs has the strong \EH property. This class of graphs is incomparable to the class of claw-free Berge graphs.

The {\em line graph $L(G)$} of a graph $G$ is the graph with vertex set $E(G)$ where $ef$ is an edge in $L(G)$ if and only if $e$ and $f$ share a vertex in $G$. While the class of line graphs is a proper subclass of the class of claw-free graphs, it is incomparable to the class of claw-free Berge graphs so the result below gives another hereditary class for which the strong \EH property holds. 
\begin{theorem}\label{thm:main-line-intro}
The class of line graphs
has the strong Erd\H{o}s-Hajnal property.
\end{theorem}
\noindent
In fact we shall require weighted versions of Theorem~\ref{thm:main-claw-free-intro} and Theorem~\ref{thm:main-line-intro}, which we state and prove in Section~\ref{sec:claw-free}.

We remark that although the strong \EH property implies the (weak) \EH property, there are graphs $H$ such that the class $\cG$ of $H$-free graphs 
satisfies the \EH property, yet the class of $\{H,H^\compl\}$-free graphs (and thus $\cG$) fails to satisfy the {\em strong} \EH property. 
The bull, a self-complementary graph, is such an example, as implied by \cite{cs2008} and the following. 
\begin{theorem}\label{er_theorem}
Let $H$ be a graph. 
\begin{itemize}
\item[$(a)$] The class of $H$-free graphs has the strong \EH property if and only if $H$ is an induced subgraph of the four-vertex path $P_4$.
\item[$(b)$] If both $H$ and its complement $H^\compl$ contain a cycle, 
then the class of $\{H,H^{\compl}\}$-free graphs does not have the strong Erd\H{o}s-Hajnal property. 
\end{itemize}
\end{theorem}
\noindent
We expect that the result above, which is proved by a simple random construction, is probably known, but we cannot find it recorded anywhere. We give the details in Section \ref{sec:negative}.

While the strong \EH property requires homogeneous pairs of linear size, if we require homogeneous pairs of only polynomial size, then some strong results are known. In \cite{ehp2000}, Erd\H{o}s, Hajnal and Pach improve results from \cite{eh1977} and show that for every graph $H$ there exists $c>0$ such that  every $H$-free graph $G$ on $n$ vertices admits a homogeneous pair $(P,Q)$ with $|P|,|Q|\ge n^c$. 
Fox and Sudakov \cite{fs2009} showed that in fact, every $H$-free graph $G$ contains either a clique of size $n^c$ or an anti-adjacent pair $(P,Q)$ with $|P|,|Q|\ge n^c$.

Finally, we would like to remark that, as shown by Bousquet, Lagoutte, and Thomass\'{e}~\cite{blt2014},
if a hereditary graph class satisfies the strong Erd\H{o}s-Hajnal property,
then it also admits a clique-independent set separation family of polynomial size (for precise definitions we refer the reader to~\cite{blt2014}).
Consequently, the latter conclusion holds for the family for $\dhookfam{k}$-free graphs for every fixed $k \geq 1$.
We point out that a conjecture of Yannakakis~\cite{Yannakakis91}, stemming from communication complexity and
asserting that \emph{every} graph admits a clique-independent set separation family of polynomial size,
was very recently disproved by G\"{o}\"{o}s~\cite{Goos15}.


\paragraph{Inspiration and methods.}
The original inspiration for our work comes from a paper of Lokshtanov, Vatshelle, and Villanger~\cite{p5-is},
who used the framework of minimal separators and potential maximal cliques to give a polynomial-time algorithm
for the (algorithmic) Independent Set problem in $P_5$-free graphs. Below we give some background to this framework without defining all the notions we make reference to; however we emphasize that the rest of the paper stands independently from this section.


A \emph{minimal separator} in a connected graph is an inclusion-wise minimal set of vertices whose deletion leaves the graph disconnected. A \emph{minimal triangulation} of a graph $G$ is an inclusion-wise minimal set of edges $F$ such that $G+F$, the graph obtained by adding the edges of $F$ to $G$, is a chordal graph. A potential maximal clique of $G$ is a set $K \subseteq V(G)$ which is a maximal clique in $G+F$ for some minimal triangulation $F$. These notions all turn out to be closely related through the notion of treewidth and tree decompositions.


Bouchitte and Todinca~\cite{todinca} studied the notion of potential maximal cliques from this perspective and showed that the natural dynamic programming algorithm for finding a maximum independent set in a graph of bounded treewidth can be modified to find such a set in time polynomial in the size of $G$ and linear in the number of potential maximal cliques in $G$. In this way,
they obtained a unified explanation for the existence of polynomial-time algorithms for the Independent Set problem in many hereditary graph classes. The work for $P_5$-free graphs~\cite{p5-is} follows the same approach, but generalises it, by showing that in $P_5$-free graphs one needs to examine only a particular (polynomially-sized) set of potential maximal cliques in the aforementioned algorithm. Subsequent work~\cite{p6-is} uses minimal separators and potential maximal cliques in a different way to develop a quasipolynomial-time algorithm for the Independent Set problem in $P_6$-free graphs.


Since the framework of minimal separators and potential maximal cliques has been successfully applied to the Independent Set problem for various hereditary graph classes,
we wished to investigate to what extent these methods are useful for problems related to the Erd\H{o}s-Hajnal conjecture. While our original proof of Theorem~\ref{thm:EH-hooks} followed this framework closely, we eventually found a simpler proof which circumvents most of the theory, although some artefacts remain.


Suppose $H$ is a fixed graph and $\cG$ is the class of $\{H, H^{\compl}\}$-free graphs. Our first observation, which is essentially expressed in Lemma~\ref{lem:SparseStructured} but also requires a result from \cite{fs2008}, is the following.
If (for a contradiction) the strong \EH property does not hold for $\cG$, then we may assume that each $n$-vertex graph $G \in \cG$ has maximum degree $o(n)$ and a minimal separator of linear size. It immediately follows that $G$ has three disjoint subsets of vertices $A,B,S$ where $S$ has linear size, where every vertex of $S$ has at least one neighbour in $A$ and $B$, and where there are no edges between $A$ and $B$. 

The next step is to use this additional structure of $G$ to form a hook i.e.\ an induced path on four vertices and to grow a $k$-vertex induced path from the third vertex of the hook. This gives an induced copy of $H_k$ and the desired contradiction. In order to obtain the $k$-vertex induced path, we apply the simple but ingenious argument of Bousquet, Lagoutte, and Thomass\'e~\cite{blt2015} that allows one to grow an arbitrarily long (but constant size) induced path in a connected  graph with sublinear maximum degree. The main work in our proof is to set up the hook so that it will not interfere with the path we wish to grow. If such a hook does not exist, then an involved analysis of vertices in $S$ and how their neighbourhoods interact reveals that $S$ has quite a restricted structure: in particular we can partition a large part of $S$ such that each pair of parts forms a homogeneous pair and such that the `quotient graph' of this partition belongs to a more restricted hereditary graph class than the one we started with. This allows us to push through an induction step which gives a linear sized homogeneous pair.

\paragraph{Structure.}
The rest of the paper is organised as follows. In Section \ref{sec:prelims}, we provide all 
necessary definitions and tools that we use throughout the paper. 
In Section~\ref{sec:claw-free}, we prove weighted versions of Theorem~\ref{thm:main-claw-free-intro} and Theorem~\ref{thm:main-line-intro}. 
In Section \ref{sec:big}, we prove Theorem~\ref{thm:main-hook}, using Theorem~\ref{thm:main-claw-free} 
and a structural result (cf.~Theorem~\ref{thm:tech-hooks}) which we prove in Section~\ref{sec:structuralTheorem}. 
As a warm-up for this technical result and to illustrate our method, we prove a simpler result in 
Section~\ref{sec:warm-up} (cf.~Theorem~\ref{thm:tech-claw-free}). 
In Section~\ref{sec:negative},  we prove Theorem~\ref{er_theorem}. 
We close the paper with some concluding remarks in Section~\ref{sec:concl}.

%% file: prelims.tex
In this section, we fix notation and terminology, and we prove a lemma which we will use several times throughout the paper.

A graph $G=(V,E)$ consists of a set $V(G):=V$ of vertices and a set $E(G):=E$ of edges, where an edge is an unordered pair of vertices. A multigraph is defined in the same way except that we allow $E(G)$ to be a multiset. A directed multigraph $D=(V,A)$ consists of a vertex set $V$ and an arc multiset $A$, where an arc is an ordered pair of vertices.  
For graphs and multigraphs we set $v(G):= |V(G)|$ and $e(G):=|E(G)|$. 
We denote the {\em complement} of a graph $G$ by $G^\compl$
where $V(G^\compl):=V(G)$ and $e \in E(G^\compl)$ if and only $e \not\in E(G)$. 
For an edge $e\in E(G)$, we write $G\sm e$ for the graph on the same vertex set as $G$ and with edge set $E(G)\sm\{e\}$.

Let $X\se V(G)$ be a subset of the vertices of a graph $G$. 
We denote by $G[X]$ the {\em induced subgraph of $G$ on $X$}
i.e.\ the graph with vertex set $X$ 
and edge set $E(G[X]):=\{uv\in E(G) \colon u,v\in X\}$. 
We 
write $G-X$ for the induced subgraph of $G$ on $V(G)\sm X$. 
Let $N_G(X) := \{u \in V(G)\sm X: uv \in E(G) \mbox{ for some } v\in X\}$ denote the {\em (open) neighbourhood of $X$} 
and let $N_G[X] := X \cup N_G(X)$ denote the {\em closed neighbourhood of $X$}. 
We omit the subscript if the graph $G$ is clear from context. 
We write $N(v):=N(\{v\})$ and $N[v]:= N[\{v\}]$. 
Furthermore, for a set $A \subseteq V(G)$ we define $N_A(X) := N(X) \cap A$ and $N_A[X] := N[X] \cap A$.
For brevity, if $X = \{x,y\}$, we write $N(x,y)$ 
instead of $N(\{x,y\})$. 
For a graph $G$ and two disjoint sets $X,Y \subseteq V(G)$, 
we denote by $E_G(X,Y)$ the set of edges of $G$ with one endpoint
in $X$ and one endpoint in $Y$.

A graph $H$ is called {\em a subgraph of $G$}, denoted by $H\se G$, if $V(H)\se V(G)$ 
and $E(H)\se E(G[V(H)])$. It is called an {\em induced subgraph of $G$} if $E(H)= E(G[V(H)])$. 
A {\em $k$-vertex path}, denoted by $P_k$, is the graph on $k$ vertices $\{v_1,\ldots, v_k\}$ 
with edge set $E(P_k):=\{v_iv_{i+1} \colon 1\leq i < k\}$. 
We call a graph {\em connected} if for every pair of vertices $x,y$ there exists a $k$-vertex path $P$ for some $k\geq 1$ 
that is a subgraph of $G$ and that contains both $x$ and $y$. 
A {\em component} in $G$ is a maximally connected subgraph of $G$. 

A {\em complete graph} is one where all possible edges are present. A {\em clique} in a graph $G$ is a subset $X\se V(G)$ such that $G[X]$ is a complete graph 
and an {\em independent set} in $G$ is a subset $X\se V(G)$ such that $G[X]$ is the empty graph. In each case, we may also refer to the subgraph $G[X]$ as a clique or independent set.

We already defined the line graph of a graph, and more generally 
if $G$ is a multigraph, the line graph $L(G)$ of $G$ is the graph with vertex set $E(G)$ where $ef$ is an edge of $L(G)$ if and only if $e$ and $f$ share a vertex in $G$. A graph $G'$ is called a line graph if it is the line graph of some (multi)graph $G$. 

Given four distinct vertices $x,a,b,c$ of a graph $G$, we say that {\em $(x;a,b,c)$ is a claw in $G$},
if $G[\{x,a,b,c\}]$ is isomorphic to a claw with $x$ being the degree-three vertex. 

\medskip
\noindent
{\bf Hooks.} 
For $k\geq 0$, recall that a $k$-hook, denoted by $\hook{k}$, is a $(k+3)$-vertex path, say on vertex set $\{v_1,\ldots,v_{k+3}\}$ and edges $v_iv_{i+1}$ for $1\leq i \leq k+2$, together with a pendant edge $v_{k+1}v_{k+4}$. The vertex $v_1$ is called the {\em active vertex} of the $k$-hook. 
Note that a $0$-hook, denoted $\hook{0}$, is the four-vertex path $P_4$, with one of its interior vertices designated as an active vertex. 

When constructing an induced $k$-hook in a graph $G$ we often start with a $0$-hook i.e.\ a copy of an induced $P_4$, and then ``grow'' a path by adding edges subsequently to the active vertex. 
The following notion is helpful. 
An {\em active $k$-hook} in a graph $G$ is a pair $(X,R)$, where $X,R \subseteq V(G)$, 
$X \cap R = \emptyset$, $G[X]$ is isomorphic to a $k$-hook, 
$G[R]$ is connected, and $N(R) \cap X$ consists of exactly one
vertex, namely the active vertex of the $k$-hook $G[X]$.

\medskip
\noindent
{\bf Modules.} 
Frequently, we will encounter sets in our graph $G$ that ``behave like a single vertex'' in the following way. 
A set $X \subseteq V(G)$ is a \emph{module} in $G$ if for every $x,y \in X$ we have $N(x) \setminus X = N(y) \setminus X$.
For a partition $V(G) = X_1\uplus X_2 \uplus\ldots \uplus X_r$  into nonempty modules $X_1, \ldots, X_r$, observe that, for every $i \neq j$, the pair $(X_i,X_j)$ is a homogeneous pair. 
For such a partition, the quotient graph $G_q$ is defined to be the graph with vertex set $\{X_1,X_2,\ldots,X_r\}$
where two sets $X_i$ and $X_j$ are connected by an edge in $G_q$ if and only if they form an adjacent pair in $G$.
Note that a quotient graph is necessarily isomorphic to some induced subgraph of $G$, namely one formed by taking exactly one vertex from every set $X_i$.

\medskip
The following simple lemma is used frequently throughout the paper.  

\begin{lemma}\label{lem:large-comp}
Let $G$ be a graph, $\mu$ be a probability measure on $V(G)$, $\delta > 0$, and $X \subseteq V(G)$ such that $\mu(X) > 3\delta$.
Then there exists either a set $P \subseteq X$ such that $(P,X \setminus P)$
is an anti-adjacent pair
and $\mu(P),\mu(X \setminus P) > \delta$, or the largest component of $G[X]$ has measure at least $\mu(X)-\delta$.
\end{lemma}
\begin{proof}
Let $C$ be the vertex set of a component of $G[X]$ such that $\mu(C)$ is maximal.
If $\mu(C) \geq \mu(X)-\delta$, then we are done. Also, if $\delta < \mu(C) < \mu(X)-\delta$,
then we are done by taking $P := C$.
In the remaining case, 
when all components of $G[X]$ have measure at most $\delta$,
we proceed as follows. We initiate $P := \emptyset$, and iterate over 
components of $G[X]$ one-by-one, putting them into the set $P$
until $\mu(P)$ exceeds $\delta$. Since every component of $G[X]$ has 
measure at most $\delta$, we have $\delta < \mu(P) \leq 2\delta$ at the end of the process. Since $\mu(X) > 3\delta$,
we have $\mu(X \setminus P) > \delta$. Furthermore, by construction, $P$ and
$X \setminus P$ form an anti-adjacent pair. This concludes the proof of the lemma.
\end{proof}

%% file: claw-free.tex
In this section, we state and prove weighted versions of Theorem~\ref{thm:main-claw-free-intro} and Theorem~\ref{thm:main-line-intro} from which those theorems immediately follow.  

A graph class $\mathcal{G}$ has \emph{the weighted strong Erd\H{o}s-Hajnal property} 
if there exists a constant $\bheps > 0$ such that every $G \in \mathcal{G}$ satisfies the following property.
For every probability measure $\mu$ on $V(G)$ satisfying $\mu(v) \leq 1-2\bheps$ for all $v \in V(G)$,  
there exists a homogeneous pair $(P,Q)$ in $G$ with $\mu(P),\mu(Q) \geq \bheps$. 
The condition that $\mu(v) \leq 1-2\bheps$ for all $v\in V(G)$ is necessary since for degenerate measures, 
where most of the mass is concentrated on one vertex only, we cannot hope to find a  
homogeneous pair of sufficient mass, for any graph $G$.

We shall prove the following two theorems.

\begin{theorem}\label{thm:main-claw-free}
The class of claw-free Berge graphs
has the weighted strong Erd\H{o}s-Hajnal property.
\end{theorem}

\begin{theorem}\label{thm:main-line}
The class of line graphs
has the weighted strong Erd\H{o}s-Hajnal property.
\end{theorem}




Theorem~\ref{thm:main-line} is an immediate corollary of the next lemma, 
where we prove that in any line graph with vertex weights, we find either an anti-adjacent pair 
or a {\em clique} of positive mass.

\begin{lemma}\label{lem:Line-hom}
Let $\bheps_1 = \frac{1}{14}$. 
Then for every graph $G$ that is a line graph of some multigraph $H$, and every probability measure $\mu$ on $V(G)$,
there exists either a clique $K$ in $G$ with $\mu(K) \geq 3\bheps_1$ or an anti-adjacent pair $(P,Q)$ in $G$
with $\mu(P),\mu(Q) > \bheps_1$.
\end{lemma}

\begin{proof}
Fix a multigraph $H$, let $G$ be its line graph, and fix a probability measure $\mu$ on $V(G)$. 
By definition, $\mu$ is a probability measure on the edges of $H$. We can find a partition of $V(H)=L \uplus R$ such that $w:=\mu(E_H(L,R)) \geq \frac 12$. Such a partition exists since for a uniformly random partition $V(H) = V_1 \uplus V_2$ the expected
value of $\mu(E_H(V_1,V_2))$ is $\frac 12$.
Let $H'$ be the bipartite subgraph of $H$ with $V(H') = V(H)$ and
$E(H') = E_H(L,R)$.
Notice that the line graph of $H'$ is an induced subgraph of $G$. 
We define a function $f:V(H) \to [0,1]$ by $f(v)=\sum_{u: uv \in E(H')} \mu(uv)$. We naturally extend $f$ to subsets of $V(H)$ by summation over the elements of the subset. Notice that $w:=f(L)=f(R) \geq \frac 12$.

Assume that $G$ has no clique of measure at least $3\bheps_1$.
Since the set of edges adjacent to a single vertex in $H'$ forms a clique in $G$, we may deduce that $f(v)< 3\bheps_1$ for every vertex in $H'$. 
We find a partition $L = L_1 \uplus L_2$ of $L$  such that 
$\frac w2 - \frac{3}{2}\bheps_1 < f(L_1),f(L_2)< \frac w2 + \frac{3}{2}\bheps_1$ in the following way.  
Pick $u\in L$, set $L_1:=\{u\}$, and add vertices from $L$ to $L_1$, one at a time, until $f(L_1)>\frac w2 - \frac{3}{2}$. 
Since we add less that $3\bheps_1$ to $f(L_1)$ each time, the upper bound on $f(L_1)$ follows. 
In the same way, we find a partition $R=R_1\uplus R_2$ of $R$ such that 
$\frac w2 - \frac{3}{2}\bheps_1 < f(R_1),f(R_2) < \frac w2 + \frac{3}{2}\bheps_1$.

For $i,j\in \{1,2\}$, let $\mu_{ij}:=\mu(E_{H'}(L_i,R_j))$. We have $\mu_{i1}+\mu_{i2}=f(L_i) \geq \frac w2 - \frac{3}{2} \bheps_1$. Likewise, $\mu_{1i}+\mu_{2i}=f(R_i) \geq \frac w2 - \frac{3}{2}\bheps_1$.

If $\mu_{12}<\bheps_1$ or $\mu_{21} < \bheps_1$, then both $\mu_{11},\mu_{22} > \frac w2 -\frac{5}{2}\bheps_1 \geq \bheps_1$. In that case, we may take 
$P=E_{H'}(L_1,R_1), Q=E_{H'}(L_2,R_2)$, since then $(P,Q)$ is an anti-adjacent pair in $G$. 
In the other case, if $\mu_{12},\mu_{21} \geq \bheps_1$,
we may take $P=E_{H'}(L_1,R_2)$ and $Q=E_{H'}(L_2,R_1)$,
and again observe that $(P,Q)$ is an anti-adjacent in $G$.
\end{proof}

We now turn to the proof of Theorem~\ref{thm:main-claw-free}. We resort to some known structural results on claw-free graphs. 
Let us first recall some standard terminology that we need.  
For a graph $G$, a pair $(T,\beta)$ is called a \emph{tree decomposition of $G$} 
if $T$ is a tree, $\beta:V(T) \to 2^{V(G)}$ is a function, and the following conditions hold:
\begin{enumerate}
\item $V(G) = \bigcup_{t \in V(T)} \beta(t)$;
\item for every $v \in V(G)$ the set $\{t : v \in \beta(t)\}$ induces a connected
subgraph of $T$;
\item for every edge $uv \in E(G)$ there exists $t \in V(T)$ such that
$u,v \in \beta(t)$.
\end{enumerate}
For a tree decomposition $(T,\beta)$ and a node $t\in V(T)$, the set $\beta(t)$ is called {\em a bag}.  
A subset $S\se V(G)$ is called {\em a separator (of $G$)} if there exist two vertices $x,y \in V(G)\sm S$ 
such that $x$ and $y$ lie in different components of $G-S$.
A {\em clique separator in $G$} is a set $S\se V(G)$ that is a separator of $G$ and such that $G[S]$ forms a clique. 
The following result on the existence of a clique separator decomposition is considered to be folklore, see e.g.~\cite{BerryPS10}. 
Since in most references it is phrased as a recursive graph decomposition instead of a tree decomposition, 
we provide the short proof for completeness. 

\begin{lemma}\label{thm:CFP-dec}
For every graph $G$ there exists a tree decomposition of $G$ where every bag induces a graph without clique separators. 
\end{lemma}
\begin{proof}
We prove the claim by induction on $|V(G)|$. In the base case, $G$ does not contain any clique separator, so we can create a tree decomposition 
$(T,\beta)$ where $T$ consists of a single node $t$ and $\beta(t):=V(G)$. 

Otherwise, let $S$ be a clique separator in $G$ such that $|S|$ is minimal. 
By minimality, there exists a component  $A$ of $G-S$ such that $N_G(A) = S$. 
Let $G_1 := G[A \cup S]$ and $G_2 := G-A$. By induction, for $i=1,2$, there exists a tree decomposition $(T_i,\beta_i)$ of the graph $G_i$.
Since $G[S]$ is a clique, and $S$ appears both in $G_1$ and $G_2$, for every $i=1,2$, there exists a bag $t_i \in V(T_i)$ such that
$S \subseteq \beta_i(t_i)$. To conclude, let $T$ be the tree formed by taking the disjoint union of $T_1$ and $T_2$ and adding the 
edge $t_1t_2$. Set $\beta(t):=\beta_i(t)$ if $t\in T_i$, and observe that $(T,\beta)$ is a suitable tree decomposition of $G$. 
\end{proof}

The following lemma provides the main reason for considering tree decompositions (with additional suitable properties) 
when studying the strong \EH property. 
It is considered folklore in the unweighted case, and we refer to it as the {\em central bag argument}. 

\begin{lemma}\label{lem:middle-bag}
Let $0 < \delta \leq \frac{1}{4}$ be a constant, let $G$ be a graph, let $\mu$ be a probability measure on $V(G)$,  
and let $(T,\beta)$ be a  tree decomposition of $G$. 
Then there exists an anti-adjacent pair $(P,Q)$ in $G$ with $\mu(P),\mu(Q) > \delta$, 
or a bag $\beta(t)$ with $\mu(\beta(t)) \geq \frac{1}{2}-\delta$.
\end{lemma}
\begin{proof}
Extend $\beta$ to subsets of nodes of $T$ by setting $\beta(S) := \bigcup_{t \in S} \beta(t)$, for every $S \subseteq V(T)$. 
We define an orientation of the edges of $T$ as follows. 
For an edge $t_1t_2 \in E(T)$, let $T_i$ be the component of $T \setminus t_1t_2$ 
that contains $t_i$, for $i=1,2$. 
Now, orient the edge $t_1t_2$ from $t_1$ to $t_2$ if 
$\mu(\beta(V(T_1))) \leq \mu(\beta(V(T_2)))$, 
and orient the edge $t_1t_2$ from $t_2$ to $t_1$ otherwise. 
Since the tree $T$ has fewer edges than nodes, there exists a node $t \in V(T)$ of out-degree zero.
For every component $C$ of $G - \beta(t)$ there exists a component 
$T_C$ of $T - \{t\}$ such that $C \subseteq \beta(V(T_C))$, by the properties of the tree decomposition. 
Therefore, $\mu(C) \leq \mu(\beta(V(T_C))) \leq \frac{1}{2}$, since the edge between
$T_C$ and $t$ is oriented towards $t$. 

If $\mu(\beta(t)) \geq \frac{1}{2}-\delta$, then we are done, so assume otherwise. 
Note that then $\mu(V(G) \setminus \beta(t)) > \frac{1}{2} + \delta \geq 3\delta$ since, by assumption, $\delta \leq \frac{1}{4}$. 
Therefore, by Lemma~\ref{lem:large-comp}, 
there is an anti-adjacent pair $(P,Q)$ in $V(G) \setminus \beta(t)$ such that $\mu(P), \mu(Q)>\delta$, 
or there is a component $C$ in $G - \beta(t)$ with $\mu(C) \geq \mu(V(G) \setminus \beta(t)) - \delta > \frac{1}{2}$. 
Since the second outcome is a contradiction, we indeed find an anti-adjacent pair of desired size. 
\end{proof}

The previous two lemmas allow us to pass to a 
linear subset $Y\se V(G)$ of a graph $G$ with the additional property that $G[Y]$ has no clique separators, 
provided that $G$ has no anti-adjacent pair of linear size. 
In light of Theorem~\ref{thm:main-claw-free}, we search for a good characterisation of claw-free Berge graphs. 
Chv\'atal and Sbihi \cite{chvatal-sbihi} show that a claw-free graph without clique separators is Berge if and only if it 
is either ``elementary'' or ``peculiar''. 
A graph is called {\em elementary} if its edges can be coloured by two colours in such a way that edges $xy$ and $yz$ 
have distinct colours whenever $x$ and $z$ are nonadjacent. 
We decide not to give the exact definition of a peculiar graph here, but rather we point out that the vertex set of a peculiar graph 
can be partitioned into nine parts that each form a clique. 
The following is then an immediate implication of Theorem 2 in \cite{chvatal-sbihi}. 

\begin{theorem}
\label{thm:CFP-char}
Let $G$ be a graph that is claw-free, Berge and that has no clique separator. 
Then $G$ is either elementary or the vertex set $V(G)$ can be partitioned into nine sets 
$V(G)=\bigcup_{i=1}^9V_i$ such that $G[V_i]$ is a clique for each $1\leq i \leq 9$.
\end{theorem}

We now resort to a characterisation of elementary graphs due to Maffray and Reed \cite{maffray-reed} 
that suits our purposes better than the original definition. We use the following terminology from \cite{maffray-reed}.  
Let $G$ be a graph. 
We call an edge a \emph{flat edge in $G$} if it does not appear in any triangle of $G$.
Let $xy$ be a flat edge in $G$, let $X, Y$ be two disjoint sets such that $X\cap V(G)= Y\cap V(G) = \emptyset$, 
and let $B=(X, Y; E_{XY})$ be a cobipartite graph, that is a graph on vertex set $X\uplus Y$, where 
$B[X]$ and $B[Y]$ form cliques, such that there is at least one edge between $X$ and $Y$ in $B$. 
We can build a new graph $G$ obtained from $G-\{x,y\}$ and $B$ by adding all possible
edges between $X$ and $N(x) \setminus \{y\}$ and between $Y$ and $N(y) \setminus \{x\}$.
We say that $G$ is \emph{augmented along $xy$}, that $x$ and $y$ are \emph{augmented}, and that \emph{$x$ is replaced by $X$} and \emph{$y$ is replaced by $Y$}.
Intuitively, we replace the vertices $x$ and $y$ by cliques, and the edge $xy$ by a (non-empty) bipartite graph. 
It is easy to see that, if $x_1y_1$ and $x_2y_2$ are independent edges in $G$, then the graph obtained by first augmenting 
$G$ along $x_1y_1$ and then the resulting graph along $x_2y_2$ is the same as if we had first augmented $x_2y_2$ and 
then $x_1y_1$. This leads to the following definition. 

An \emph{augmentation of a graph $G$} is a graph $G'$ that is obtained by augmenting $G$ 
along the edges of some matching of flat edges in $G$.

\begin{theorem}[\cite{maffray-reed}]\label{thm:Elem-char}
A graph $G$ is elementary if and only if it is an augmentation of a line graph of a bipartite multigraph.
\end{theorem}
We now prove the equivalent of Lemma \ref{lem:Line-hom} for elementary graphs.
\begin{lemma}\label{lem:CFP-hom}
Let $\bheps_2 = \frac{1}{28}$.
Then, for every graph $G$ that is an elementary graph, and every probability measure $\mu$ on $V(G)$,
there exists either a clique $K$ in $G$ with $\mu(K) \geq 3\bheps_2$ or an anti-adjacent pair $(P,Q)$ in $G$
with $\mu(P),\mu(Q) \geq \bheps_2$.
\end{lemma}
\begin{proof}
Let $G'$ be a line graph of a bipartite multigraph $B$ such that $G$ is an augmentation of $G'$, 
which exists by Theorem~\ref{thm:Elem-char}. 
We define a probability measure $\mu'$ on $V(G')$ in the natural way by 
setting $\mu'(x):=\mu(X)$ if $x\in V(G')$ was augmented and replaced by $X$, 
and $\mu'(x):=\mu(x)$ otherwise. 


We apply Lemma \ref{lem:Line-hom} to $G'$ to find either an anti-adjacent pair $(P,Q)$ 
in $G'$ with $\mu'(P),\mu'(Q) \geq \bheps_1 = 2\bheps_2$, or a clique, say on vertex set $K$, such that $\mu'(K)\geq 3\bheps_1$. 
In the first outcome, note that some vertices of $P$ and $Q$ may have been replaced by cliques in the augmentation $G$, 
say $x_1,\ldots,x_p\in P$ are replaced by $X_1,\ldots, X_p$, and $y_1,\ldots,y_q$ are replaced by $Y_1,\ldots, Y_q$. 
Set $P_G:= (P\sm\{x_1,\ldots,x_p\})\cup X_1\cup\ldots \cup X_p$ and $Q_G:= (Q\sm\{y_1,\ldots,y_q\})\cup Y_1\cup\ldots\cup Y_q$, 
and note that $(P_G,Q_G)$ is an anti-adjacent pair in $G$ with $\mu(P_G),\mu(Q_G) \geq 2\bheps_2$. 
In the second outcome, 
we consider two cases, depending on $|K|$.
If $|K| \leq 2$, then the heaviest vertex of $K$ corresponds to a clique in $G$ of measure at least $3\bheps_1/2 = 3\bheps_2$.
Otherwise, if $K$ consists of at least three vertices, then none of its edges is flat, and therefore it remains a clique of measure at least $3\bheps_1$ in $G$.
\end{proof}

We now deduce theorem \ref{thm:main-claw-free} as a corollary. 
\begin{proof}[Proof of Theorem \ref{thm:main-claw-free}]
We prove the weighted strong Erd\H{o}s-Hajnal property with constant $\bheps_3 = \frac{1}{58}$. 
Let $G$ be a graph that is claw-free and Berge, and let $\mu$ be a probability measure on 
$V(G)$ such that for every vertex $v$ of $G$ we have that $\mu(v)\leq 1-2\delta_3$. 
In fact, we may assume that $\mu(v)< \delta_3$ for all $v\in V(G)$. 
For let $v\in V(G)$ be a vertex such that $\delta \leq \mu(v)\leq 1-2\delta_3$. Then   
$\mu(N(v))\geq \delta_3$ or $\mu(V(G)\sm N[v])\geq \delta_3$. 
That is, either $(\{v\},N(v))$ is an adjacent pair in $G$ of sufficient mass, or 
$(\{v\},V(G)\sm N[v])$ is an anti-adjacent pair in $G$ of sufficient mass, and we are done.  

Let $(T,\beta)$ be a tree-decomposition of $G$ such that every bag induces a subgraph of $G$ without a clique separator, 
which exists by Theorem \ref{thm:CFP-dec}. 
By Lemma~\ref{lem:middle-bag}, either there is an anti-adjacent pair $(P,Q)$ in $G$ with $\mu(P),\mu(Q) > \bheps_3$ and we are done,
or there is a bag $Y$ with $\mu(Y) \geq \frac{1}{2}-\bheps_3$. 
In the second case, we apply Theorem~\ref{thm:CFP-char} to $G[Y]$ to infer that $G[Y]$ is either elementary or its vertex set can be partitioned into nine cliques. 
In the latter case, $G[Y]$ contains a clique of measure at least $\frac{\mu(Y)}{9} \geq 3\bheps_3$ and we are done. 
If $G[Y]$ is elementary, then by Lemma~\ref{lem:CFP-hom}, $G[Y]$ contains an anti-adjacent pair $(P,Q)$ with $\mu(P),\mu(Q) > \frac{1}{28}\mu(Y)$
or a clique $K$ with $\mu(K) \geq \frac{3}{28}\mu(Y)$. In both cases we are done by the choice of $\bheps_3$, as $\frac{1}{28}\mu(Y) \geq \frac{0.5-\bheps_3}{28} = \bheps_3$.
\end{proof}

%% file: big-picture.tex
In this section, we state two technical lemmas and show how Theorem~\ref{thm:main-hook} can be derived from them. The lemmas will be proved in Sections~\ref{sec:warm-up} and~\ref{sec:structuralTheorem}.

Fix $k\geq 1$ and let $\cG:= \{G : G \mbox{ is } \dhookfam{k} \mbox{-free}\}$ i.e. the class of all graphs $G$ that are 
$\{\dhook{\ell},(\dhook{\ell})^{\mbox{c}}\}$-free, for all $\ell\geq k$. To prove Theorem \ref{thm:main-hook}, we show that $\cG$ has the {\em strong} \EH property. That is, we need to find a $\delta>0$ such that every $G\in\cG$ contains a homogeneous pair $(P,Q)$ with $|P|,|Q|\geq \delta\cdot v(G)$. 
Similarly as in~\cite{blt2015}, our starting point is to pass down to an induced subgraph that is very sparse or very dense. 
The {\em edge density} of a graph $G$ is the fraction $e(G)/\binom{v(G)}{2}$. 
The following is due to Fox and Sudakov \cite{fs2008}, improving an earlier result of R\"odl \cite{r1986}. 
\begin{theorem}\label{thm:reg}
For every $0 < \degeps < 1/2$ and every graph $H$ on at least two vertices there exists a constant 
$\delta=\delta(\degeps,H)$ such that
every $H$-free graph on $n$ vertices contains an induced subgraph on at least  
$\delta n$ vertices with edge density either at most $\degeps$ or at least $1-\degeps$.
\end{theorem}

%
%

In case the induced subgraph is particularly sparse 
we find a special structure within it. 
%
Let $G$ be a graph, and let $\cS=(A,B,S)$ be a triple of non-empty subsets of $V(G)$. We call the pair $(G,\cS)$ an {\em $\eps$-structured pair} if 
\begin{enumerate}
\item $A \uplus B \uplus S = V(G)$, i.e., the sets $A, B, S$ form a partition of $V(G)$;
\item $G[A]$ and $G[B]$ are connected; 
\item $N(A) = N(B) = S$, in particular, there is no edge between $A$ and $B$; and 
\item for every $v \in V(G)$ it holds that $|N_S[v]| \leq \eps |S|$. 
\end{enumerate}
Note that if $(G,\cS)$ is $\eps$-strucutred, then it is $\eps'$-structured for every $\eps'\geq\eps$.

\begin{lemma}\label{lem:SparseStructured}
Fix $0<\eps<\frac{1}{10}$ and let $G$ be a graph on $n$ vertices  
such that every vertex has at most $\eps n$ neighbours. 
Then 
\begin{itemize}
\item[$(a)$] there exists a homogeneous pair $(P,Q)$ in $G$ with $|P|, |Q|\geq  n/10$; or
\item[$(b)$] there exist subsets $A,B,S\se V(G)$ such that $|S|\geq n/10$ and the pair $(G[A\cup B\cup S],(A,B,S))$ is a $10\eps$-structured pair. 
\end{itemize}
\end{lemma}

\begin{proof}
Assume that there is no homogeneous pair $(P,Q)$ with $|P|, |Q| \geq n/10$ in $G$, 
and let $G_1$ be the largest component of $G$.  
By Lemma~\ref{lem:large-comp}, $G_1$ has at least $9n/10$ vertices. 
%
Pick an arbitrary vertex $x_A$ in $G_1$, and set $A:=\{x_A\}$. Now add vertices one by one to $A$, keeping $A$ connected, until $|N[A]|$ exceeds $n/2$. 
Then, $|N[A]|\leq n/2 +\eps n< 3n/5$, 
since we add at most $\eps n$ vertices to $N[A]$ in each step. 
Thus, $G_1- N[A]$ has at least $3n/10$ vertices. 
Let $B$ be the largest component in $G_1- N[A]$. 
By Lemma~\ref{lem:large-comp}, we may assume that $|B|\geq n/10$ 
and therefore we must also have that $|A|< n/10$ (otherwise $(A,B)$ is an anti-adjacent pair of sufficient size). 
Thus, $|N(A)|\geq 2n/5$. 
Furthermore, $|N(A)\sm N[B]|< n/10$, since otherwise, $(B, N(A)\sm N[B])$ is an anti-adjacent pair of sufficient size. 
Setting $S:= N(A)\cap N(B)$, we see that $|S|\geq 3n/10$ by the above discussion, and 
for every $v\in V(G)$ we have that $|N_S(v)|\leq |N(v)|\leq \eps n\leq 10\eps |S|$. Furthermore it is easy to see that in $G' = G[A \cup B \cup S]$ we have $N(A) = N(B) = S$ and that $G'[A]$ and $G'[B]$ are connected.
Thus the pair $(G[A\cup B\cup S],(A,B,S))$ is a $10\eps$-structured pair. 
\end{proof}

The following theorem is the crucial step in our proof. It states that within an $\eps$-structured pair $(G,(A,B,S))$, 
we either find the desired homogeneous pair of linear size, 
or a very structured subset $\hat{S}\se S$ of linear size,    
or an $\ell$-hook for some $\ell\geq k$ which we can potentially extend to a double $\ell$-hook. 
Recall that an {\em active $\ell$-hook} in a graph $G$
is a pair $(X,R)$, where $X,R \subseteq V(G)$,
$X \cap R = \emptyset$, $G[X]$ is isomorphic to an $\ell$-hook,
$G[R]$ is connected, and $N(R) \cap X$ consists of exactly one
vertex, being the active vertex of the $\ell$-hook $G[X]$.

\begin{theorem}\label{thm:tech-hooks}
For every $k \geq 0$, there exists a 
constant $\eps_0$ such that for every $0<\eps\leq\eps_0$ and  
in every $\eps$-structured graph $(G,(A,B,S))$ there exists either 
\begin{enumerate}
\item an anti-adjacent pair $(P,Q)$ in $G$ with $P,Q \subseteq S$, $|P|,|Q| \geq \eps |S|$; or 
\item an active $\ell$-hook $(X,R)$ in $G$ with $\ell \geq k$, $R \subseteq S$, and $|R| \geq 2\eps |S|$; or
\item a subset $\hat{S} \subseteq S$ with $|\hat{S}| \geq  |S|/5$ 
and a partition $\hat{S} = S_1 \uplus S_2 \uplus \ldots \uplus S_m$, for some $m\geq 2$, 
such that
\begin{enumerate}
\item $|S_i|\leq \eps |S|$ for every $1\leq i \leq m$; 
\item every set $S_i$ is a module of $G[\hat{S}]$; and 
\item the quotient graph of this partition of the vertex set
of $G[\hat{S}]$ is a claw-free Berge graph.
\end{enumerate}
\end{enumerate}
\end{theorem}

We delay the proof of this theorem until Section \ref{sec:structuralTheorem}. 
Informally, the idea is as follows. 
Let  $(G,(A,B,S))$ be an $\eps$-structured pair and assume that $G$ does not contain a homogeneous pair of linear size or an active $\ell$-hook. 
After some filtering, we partition the vertices in $S$ into equivalence classes 
according to their neighbourhoods in $A\cup B$. 
Assuming certain subgraphs like the hook are forbidden in $G$,
it turns out that edges and non-edges between pairs of vertices in $S$ correspond to a certain behaviour of the neighbourhoods of those vertices in $A$ and $B$.
 This allows us to deduce that 
the equivalence classes of the partition on $S$ are in fact modules. Furthermore, the quotient graph turns out to have an even more restricted structure in terms of the induced subgraphs that are forbidden. 

We believe that these methods can be of further use to approach similar problems. 
Since the proof of Theorem \ref{thm:tech-hooks} is rather technical, we  present the following as a warm-up in Section \ref{sec:warm-up} to illustrate our methods, although we will need many of the lemmas from Section \ref{sec:warm-up} later.  

\begin{theorem}\label{thm:tech-claw-free}
For every $1/10$-structured graph $(G,(A,B,S))$ 
such that $G$ is both claw-free and $C_5$-free,
there exists 
a subset $\hat{S} \subseteq S$ with $|\hat{S}| \geq |S|/5$
and a partition $\hat{S} = S_1 \uplus S_2 \uplus\ldots \uplus S_\ell$
such that:
\begin{enumerate}
\item every set $S_i$ is contained in a neighbourhood of some vertex in $A$;
\item every set $S_i$ is a module of $G[\hat{S}]$;
\item the quotient graph of this partition of the vertex set
of $G[\hat{S}]$ is a line graph of a triangle-free graph.
\end{enumerate}
\end{theorem}

We now prove our main result. 

\begin{proof}[Proof of Theorem \ref{thm:main-hook}]
Note that if the theorem holds for $k=r$ then it holds for all $1 \leq k \leq r$. and so it is sufficient to prove the theorem for all $k \geq 2$.
Thus, fix $k\geq 2$ and let $\cG:= \{G : G \mbox{ is } \dhookfam{k}\mbox{-free}\}$ and set $\eps_0=\eps_0(k)$ to be the constant from Theorem \ref{thm:tech-hooks}. We shall prove the following claim.

\begin{claim} 
Suppose $G \in \cG$ has $n$ vertices and maximum degree $\eps_0 n / 100$.
Then either $G$ has a homogenous pair $(P,Q)$ where 
$|P|,|Q| \geq \eps_0 n / 3000$ or we can find an active $\ell$-hook $(X,R)$ for some $\ell \geq k$, where $|R|\geq \eps_0 n/ 50$.
\end{claim}

\begin{proof}[Proof of Claim]  
By Lemma \ref{lem:SparseStructured},  
either there is a homogeneous pair $(P,Q)$ in $G$ with $|P|, |Q|\geq n/10$ (in which case we are done) or there is an $\frac{\eps_0}{10}$-structured pair $(G [A\cup B\cup S], (A,B,S))$ with $|S|\geq n/10$. 
 
Set $G_1:= G[A\cup B\cup S]$. 
By Theorem~\ref{thm:tech-hooks}, there is either
1.~an anti-adjacent pair $(P,Q)$ in $S$ with $|P|,|Q| \geq \frac{\eps_0}{10} |S| \geq \frac{\eps_0}{100}n$ (in which case we are done); 
or 2.~an active $\ell$-hook $(X,R)$ with $\ell\geq k$, $R\se S$ and $|R|\geq 2 \frac{\eps_0}{10} |S|$; 
or 3.~a subset $\hat{S} \se S$ with $|\hat{S}| \geq |S|/5$ and a partition 
$\hat{S} = S_1 \uplus S_2 \uplus \ldots \uplus S_m$, for some $m\geq 2$, such that 
\begin{itemize}
\item[$(a)$] $|S_i|\leq \frac{\eps_0}{10} |S| \leq \frac{\eps_0}{2}|\hat{S}|$ for every $1\leq i \leq m$; 
\item[$(b)$] every set $S_i$ is a module of $G[\hat{S}]$; and 
\item[$(c)$] the quotient graph of this partition of the vertex set 
of $G[\hat{S}]$ is a claw-free Berge graph. 
\end{itemize}

In the third outcome, we consider the quotient graph $G_q$ that has vertex set 
$V_q:=\{S_i : 1\leq i \leq m\}$, and where $S_iS_j$ forms an edge in $G_q$ if and only if $(S_i,S_j)$ is an adjacent pair.
We define a probability measure $\mu$ on $V_q$ in the natural way by setting $\mu(S_i):=|S_i|/|\hat{S}|$. 
Note that, by Property $(a)$, $\mu(S_i)\leq \eps_0/2$ for every vertex $S_i$ in $G_q$. 
By Property $(c)$, the graph $G_q$ is claw-free and Berge. 
We now invoke Theorem \ref{thm:main-claw-free} to see that 
there is either a homogeneous pair $(P_q, Q_q)$ in $G_q$ with 
$\mu(P_q), \mu(Q_q)\geq \delta_{cB}$, or there is a vertex $S_i\in V_q$ 
with $\mu(S_i)\geq 1-2\delta_{cB}$, where $\delta_{cB}$ is a 
constant that can be taken to be $\frac{1}{58}$ (see Section \ref{sec:claw-free}). 
Since for every $1\leq i \leq m$ we have that $\mu(S_i)\leq \eps_0 /2 \leq 1-2\delta_{cB}$ 
the first outcome must hold for $G_q$. 
Consider the sets $P:=\bigcup_{S_i\in P_q} S_i$ and $Q:=\bigcup_{S_i\in Q_q} S_i$. Note that since the $S_i$'s are modules and $(P_q,Q_q)$ is a homogeneous pair in $G_q$, then $(P,Q)$ is a homogeneous pair in $G[\hat{S}]$ and hence in $G$.
Furthermore $|P|,|Q|\geq \delta_{cB} |\hat{S}| \geq \frac{\delta_{cB}}{50} \cdot  n \geq \frac{n}{3000}$, giving us the homogeneous pair of the desired size. 

Thus we may assume the second outcome holds, where we find an active $\ell$-hook $(X,R)$, for some $\ell \geq k$, such that $R\se S$ and $|R|\geq 2\frac{\eps_0}{10} |S| \geq \frac{\eps_0}{50}n$ as required. 
\end{proof}

To prove the theorem, we must show that there exists a constant $\delta>0$ such that every $G\in\cG$ contains a homogeneous pair $(P,Q)$ with $|P|,|Q|\geq \delta\cdot v(G)$. 
Fix $\eps = \eps_0^2 / 40000$, set $\delta_0 = \delta(\eps,H_k^2)$ where $\delta(\cdot, \cdot)$ is the constant from Theorem~\ref{thm:reg} and set $\delta = \delta_0\eps_0^2 / 10^7$. 
Thus since both $G$ and $G^\compl$ are $H_k^2$-free, Theorem~\ref{thm:reg} implies that either $G$ or $G^\compl$ contains an induced subgraph, say $G_0$, 
on at least $\delta_0\cdot v(G)$ vertices with edge density at most $\eps$. Assume without loss of generality that $G_0$ is an induced subgraph of $G$.
By a simple averaging argument, $G_0$ contains an induced subgraph $G_1$, with $v(G_1) \geq v(G_0)/2 \geq \delta_0\cdot v(G) / 2$ and where $G_1$ has maximum degree at most $4\eps\cdot v(G_1)$.

By the claim, either $G_1$ (and hence $G$) has a homogeneous pair of size at least $\eps_0\cdot v(G_1) / 3000 \geq \delta\cdot v(G)$ (and we are done) or $G_1$ has an active $\ell$-hook $(X,R)$ for some $\ell \geq k$ with $|R|\geq \eps_0\cdot v(G_1) / 50$. 
That is, $G[X]$ is isomorphic to an $\ell$-hook, say with active vertex $x$, and $N(R)\cap X=\{x\}$. 
Set $R':= R\sm N(x)$. Note that $|R'| \geq |R| - \Delta(G_1) \geq \eps_0\cdot v(G_1) / 100$.

We can now apply the claim to $G_2 = G[R']$ since $\Delta(G_2) \leq \Delta(G_1) \leq 4\eps\cdot v(G_1) \leq \eps_0|R'|/100$. Thus either $G_2$ has a homogeneous pair of size at least $\eps_0|R'|/3000 > \delta\cdot v(G)$ (and we are done) or $G_2$ has an active $\ell$-hook $(X^*,R^*)$ for some $\ell \geq k \geq 2$.
Let $P_x$ be the shortest path in the graph $G[R]$ between the vertex $x$ (the active vertex of the active $\ell$-hook 
$(X,R)$) and the set $X^*$, and suppose it meets $X^*$ at the vertex $y$.
Such a path certainly exists since $N(R)\cap X=\{x\}$ and $R$ is connected, and it must be an induced path (since it is a shortest path). 
Now, one easily sees that $X\cup V(P_x)\cup X^*$ contains an induced copy of a double $\ell'$-hook, 
for some $\ell'\geq \ell$ (no matter where $y$ is in $X^*$!). 
(To see that this copy is induced, note that by definition, there are no edges between 
$X$ and $R$ other than those between $x$ and $N_R(x)$.) 
But this is a contradiction to having $G \in \cG$. 
\end{proof}

%% file: tech-claw-free.tex
\newcommand{\Req}{\ensuremath{R^=}}
\newcommand{\Rinc}{\ensuremath{R^{\neq}}}
\newcommand{\Rsub}{\ensuremath{R^\subsetneq}}
\newcommand{\Rsup}{\ensuremath{R^\supsetneq}}
\newcommand{\ReqA}{\ensuremath{R^=_A}}
\newcommand{\ReqB}{\ensuremath{R^=_B}}

\newcommand{\Esub}{\ensuremath{E^\subsetneq}}
\newcommand{\EeqA}{\ensuremath{E^=_A}}
\newcommand{\EeqB}{\ensuremath{E^=_B}}

The aim of this section is to provide a proof of Theorem~\ref{thm:tech-claw-free}
that will serve as a warm-up before proving the main technical step of this paper, namely Theorem~\ref{thm:tech-hooks}.
The proofs of Theorems~\ref{thm:tech-claw-free} and~\ref{thm:tech-hooks} follow the same general outline,
while the technical details in this section are much simpler. 
To exhibit the similarities between the proofs, we use nearly the same subsection structure in this section and the next one,
even though here some subsections will consist only of a single simple observation.

Let us fix a $1/10$-structured pair $(G,(A,B,S))$
such that $G$ is claw-free and $C_5$-free. 
Define a binary relation $\Req$ on $S$ as $\Req(x,y)$ if
and only if $N_{A \cup B}(x) = N_{A \cup B}(y)$ and we note that this is an equivalence relation.
Our approach consists of the following steps:
\begin{enumerate}
\item We start with filtering out vertices $x \in S$ that have large neighbourhood in $A$ or in $B$. Since
every vertex $p \in A \cup B$ satisfies $|N_S(p)| \leq |S|/10$, a standard averaging argument shows that the number of such vertices
is small.
\item Second, for every $x \in S$, we study nonedges inside neighbourhoods $N_A(x)$; such nonedges turn out to be good starting
points to construct either a claw (in the case of Theorem~\ref{thm:tech-claw-free}) or a hook (in the case of Theorem~\ref{thm:tech-hooks}).
\item Then, for every two vertices $x,y \in S$, we investigate how the neighbourhoods $N_A(x)$ and $N_A(y)$ differ,
  depending on whether $xy$ is an edge or a nonedge. Intuitively, we want to prove that if $xy \in E(G)$, then the neighbourhoods
  in $A$ and $B$ cannot change much, while if $xy \notin E(G)$, then they should change much or not at all.
\item We then collect the main properties we need from the aforementioned steps in the definition of a \emph{nice}  $\degeps$-structured pair.
We prove that the relevant $\degeps$-structured pair is nice, both in the proof of Theorem~\ref{thm:tech-claw-free} and Theorem~\ref{thm:tech-hooks}.
In this section we show that it is sufficient for the relevant $\degeps$-structured pair to be nice in order to find a large set $\hat{S} \subseteq S$, such that if we
restrict the relation $\Req$ to $\hat{S}$, the equivalence classes of this relation
form a decomposition of $G[\hat{S}]$ into modules.
\item Finally, we show that in the case of Theorem~\ref{thm:tech-claw-free}, the quotient graph of the aforementioned
decomposition is diamond-free; this, together with being claw-free, implies that the quotient graph is in fact a line graph
of a triangle-free graph, concluding the proof of Theorem~\ref{thm:tech-claw-free}.
\end{enumerate}

In the proofs of Theorems~\ref{thm:tech-claw-free} and~\ref{thm:tech-hooks},
if a statement is accompanied with a sign (\ABsym), then we also claim that the same statement holds
with the roles of $A$ and $B$ swapped.

\subsection{Filtering step}

Let $S_A = \{x \in S: A \subseteq N(x)\}$ and similarly define $S_B$.
A standard averaging argument shows the following:
\begin{claim}
$|S_A|,|S_B| \leq |S|/10$.
\end{claim}
\begin{proof}
Consider the following random experiment: independently and uniformly at random pick a vertex $p \in A$ and $x \in S$.
Since every vertex in $A$ is adjacent to at most $|S|/10$ vertices of $S$, the probability that $px \in E(G)$ is at most
$1/10$. On the other hand, once $x \in S_A$, we have $px \in E(G)$ regardless of the choice of $p$. Consequently, the probability
that $x \in S_A$ is at most $1/10$.
\end{proof}
Define now $S' = S \setminus (S_A \cup S_B)$ and $G' = G \setminus (S_A \cup S_B)$.
Since $|S_A \cup S_B| \leq |S|/5$, we have that $(G',(A,B,S'))$ is an $1/8$-structured pair.

By restricting ourselves to the structured pair $(G',(A,B,S'))$, it suffices to prove the conclusion of Theorem~\ref{thm:tech-claw-free}
with stronger condition $|\hat{S}| \geq |S|/4$, but with the additional assumption
\begin{equation}\label{eq:cf:filter}
\forall_{x \in S} (N_A(x) \subsetneq A) \wedge (N_B(x) \subsetneq B).
\end{equation}
To simplify the notation, 
in the rest of this section we assume that the input structured graph is only $1/8$-structured, but satisfies already~\eqref{eq:cf:filter}.

\subsection{Neighbourhoods in $A \cup B$}

\subsubsection{Nonedges inside a neighbourhood in $A$}

In the case of claw-free graphs, there are simply no edges inside neighbourhoods in $A$.
\begin{claim}[\ABsym]\label{lem:cf:Anei}
For every $x \in S$ the set $N_A(x)$ is a clique.
\end{claim}
\begin{proof}
Assume the contrary, let $p,q \in N_A(x)$, $p \neq q$, and $pq \notin E(G)$.
Let $z \in N_B(x)$ be any vertex (it exists since $N(B) = S$). 
Then $(x;p,q,z)$ is a claw in $G$, a contradiction.
\end{proof}

\subsubsection{Neighbourhoods along a nonedge in $S$}

\begin{claim}[\ABsym]\label{lem:cf:Snonedge}
For every $x,y \in S$ with $x \neq y$, $xy \notin E(G)$,
there is no edge between $N_A(x) \cap N_A(y)$
and $A \setminus N_A(x,y)$.
\end{claim}
\begin{proof}
Assume the contrary, let $p \in N_A(x) \cap N_A(y)$ and $q \in A \setminus N_A(x,y)$ with $pq \in E(G)$.
Then $(p;x,y,q)$ is a claw in $G$, a contradiction.
\end{proof}

\subsubsection{Neighbourhoods along an edge in $S$}

\begin{claim}\label{lem:cf:Sedge}
For every $xy \in E(G[S])$, either $N_A(x) \setminus N_A(y)$
or $N_B(x) \setminus N_B(y)$ is empty.
\end{claim}
\begin{proof}
Assume the contrary, let $p_\Gamma \in N_\Gamma(x) \setminus N_\Gamma(y)$ for $\Gamma \in \{A,B\}$.
Then $(x;y,p_A,p_B)$ is a claw in $G$, a contradiction.
\end{proof}

\subsection{Niceness of an $\degeps$-structure and its corollaries}\label{sec:nice}

In the following definition, we extract some properties of the $\eps$-structured pair $(G,(A,B,S))$
that were proven in Claims~\ref{lem:cf:Anei}, \ref{lem:cf:Snonedge}, and~\ref{lem:cf:Sedge},
and then show what can be deduced from these properties only.
Exactly the same properties will be proven in the next section, in the more general setting
of Theorem~\ref{thm:tech-hooks}, and hence we will be able to reuse the statements obtained here.

\begin{definition}\label{def:nice}
An $\degeps$-structured pair $(G,(A,B,S))$ is called \emph{nice}
if the following holds:
\begin{description}
\item[(NE1)] for every $x \in S$ we have $A \not\subseteq N(x)$ and $B \not\subseteq N(x)$;
\item[(NE2)] (\ABsym) for every $x,y \in S$ with $x \neq y$ and $xy \notin E(G)$, if $N_B(x) \neq N_B(y)$, then there is no edge between
$N_A(x) \cap N_A(y)$ and $A \setminus N_A(x,y)$;
\item[(NE3)] for every $x,y \in S$ with $x \neq y$, $xy \notin E(G)$, and $N_A(x) \subsetneq N_A(y)$, the sets
$N_A(x)$ and $N_A(y) \setminus N_A(x)$ are fully adjacent;
\item[(E1)] for every $x,y \in S$ such that $xy \in E(G[S])$, either $N_A(x) \setminus N_A(y) = \emptyset$
or $N_B(x) \setminus N_B(y) = \emptyset$.
\end{description}
\end{definition}

Note that we have used here the notation (\ABsym), denoting that the particular condition is required to hold also
with the roles of $A$ and $B$ swapped. Whenever $\eps$ is unimportant for the analysis we shall drop it from the notation and speak only of a {\em (nice) structured pair}. 

Let us now formally verify that the considered structured pair $(G,(A,B,S))$ is nice.
\begin{claim}\label{lem:cf:nice}
The structured pair $(G,(A,B,S))$ is nice.
\end{claim}
\begin{proof}
Property (NE1) is equivalent to~\eqref{eq:cf:filter},
Property (NE2) is strictly weaker than the statement of Claim~\ref{lem:cf:Snonedge},
Property (NE3) is a special case of the statement of Claim~\ref{lem:cf:Anei},
while Property (E1) is exactly the statement of Claim~\ref{lem:cf:Sedge}.
\end{proof}

We start our analysis of nice structured graphs with the following observation.
\begin{lemma}\label{lem:nice:inc}
If a structured pair $(G,(A,B,S))$ satisfies Properties~(NE1) and~(NE2), then for every two distinct vertices $x,y \in S$ with $xy \notin E(G)$
we have $N_A(x) = N_A(y)$ if and only if $N_B(x) = N_B(y)$.
\end{lemma}
\begin{proof}
Assume by contradiction that for some $x,y \in S$ with $x \neq y$ and $xy \notin E(G)$ we have
$N_A(x) = N_A(y)$ but $N_B(x) \neq N_B(y)$. By Property~(NE2), there is no edge between $N_A(x) \cap N_A(y) = N_A(x)$
and $A \setminus N_A(x,y) = A \setminus N_A(x)$. However, by Property~(NE1) and the assumption $N(A)=S$,
both $N_A(x)$ and $A \setminus N_A(x)$ are nonempty. This contradicts the connectivity of $G[A]$.
\end{proof}

We now move to a deeper study of the situation treated in Property (NE3).
\begin{lemma}[\ABsym]\label{lem:nice:stack}
If $(G,(A,B,S))$ is a nice structured pair, then there do not exist three distinct vertices $x,y,z \in S$
with $xy,yz \notin E(G)$ and $N_A(x) \subsetneq N_A(y) \subsetneq N_A(z)$.
\end{lemma}
\begin{proof}
Assume the contrary, and let $x,y,z$ be as in the statement. Let $p$ be any vertex of $N_A(x)$ and $q$ be any vertex
of $N_A(z) \setminus N_A(y)$. By Lemma~\ref{lem:nice:inc} applied to the pair $(x,y)$, we
have $N_B(x) \neq N_B(y)$ since $N_A(x) \neq N_A(y)$.
By Property (NE2) applied to the pair $(x,y)$, we have $pq \notin E(G)$, since $p \in N_A(x) = N_A(x) \cap N_A(y)$
and $q \in N_A(z) \setminus N_A(y) \subseteq A \setminus N_A(x,y)$. However, Property (NE3) applied to the pair $(y,z)$ implies
that $pq \in E(G)$, a contradiction.
\end{proof}

Recall that we have defined the relation $\Req$ on the set $S$ as $\Req(x,y)$ if and only if $N_{A \cup B}(x) = N_{A \cup B}(y)$.
We now introduce a number of other binary relations on the set $S$ that describe the relation between neighbourhoods in $A \cup B$.
For two vertices $x,y \in S$ we have
\begin{description}
\item[$\Rinc(x,y)$] if and only if $N_A(x)$ and $N_A(y)$ are incomparable with respect to inclusion, and $N_B(x)$ and $N_B(y)$ are incomparable with respect to inclusion;
\item[$\ReqA(x,y)$] if and only if $N_A(x) = N_A(y)$ and $N_B(x) \neq N_B(y)$;
\item[$\ReqB(x,y)$] if and only if $N_B(x) = N_B(y)$ and $N_A(x) \neq N_A(y)$;
\item[$\Rsub(x,y)$] if and only if $N_A(x) \subsetneq N_A(y)$ and $N_B(x) \supsetneq N_B(y)$;
\item[$\Rsup(x,y)$] if and only if $N_A(x) \supsetneq N_A(y)$ and $N_B(x) \subsetneq N_B(y)$.
\end{description}
Observe that the relations $\Req$, $\Rinc$, $\ReqA$, and $\ReqB$ are symmetric, while $\Rsub$ and $\Rsup$ are strongly antisymmetric, and $\Rsub(x,y)$ if and only if $\Rsup(y,x)$. Furthermore, all six defined relations are pairwise disjoint.

Lemma~\ref{lem:nice:stack} implies that, along nonedges in $S$, the neighbourhoods in $A$ cannot create chains with respect
to inclusions. As a corollary, we can obtain the following:
\begin{lemma}\label{lem:nice:getS}
If $(G,(A,B,S))$ is a nice structured pair, then there exists a set $\hat{S} \subseteq S$ of size at least $|S|/4$
such that for every $x,y \in \hat{S}$ with $x \neq y$ and $xy \notin E(G)$, either $\Req(x,y)$ or $\Rinc(x,y)$.
\end{lemma}
\begin{proof}
Consider an auxiliary directed multigraph $G_A$ defined as follows: we take $V(G_A) = S$ and for every $x,y \in S$ with $x \neq y$
and $xy \notin E(G)$ we add an arc $(x,y)$ if $N_A(x) \subsetneq N_A(y)$. Let $S_A^+$ be the set of vertices of $S$
that have positive out-degree in $G_A$, and let $S_A^-$ be the set of vertices of $S$ that have positive in-degree.
Symmetrically, define $G_B$ and sets $S_B^+$ and $S_B^-$.
For $\alpha, \beta \in \{+,-\}$, define $S^{\alpha\beta} = S \setminus (S_A^\alpha \cup S_B^\beta)$.

Lemma~\ref{lem:nice:stack} implies that $S_A^+ \cap S_A^- = \emptyset$ and $S_B^+ \cap S_B^- = \emptyset$, which in turn implies
that $S^{++} \cup S^{+-} \cup S^{-+} \cup S^{--} = S$. Consequently, by setting $\hat{S}$ to be the largest of the sets
$S^{\alpha\beta}$, we have $|\hat{S}| \geq |S|/4$. The definition of the sets $S^{\alpha\beta}$ ensures
that $G_A[\hat{S}]$ and $G_B[\hat{S}]$ are arcless, that is, for every $x,y \in \hat{S}$ with $xy \notin E(G)$
it cannot happen that $N_A(x) \subsetneq N_A(y)$ or $N_B(x) \subsetneq N_B(y)$.
However, Lemma~\ref{lem:nice:inc} ensures that once $N_{A \cup B}(x) \neq N_{A \cup B}(y)$ for some $x,y \in S$ with $xy \notin E(G)$,
then both $N_A(x) \neq N_A(y)$ and $N_B(x) \neq N_B(y)$,
and, consequently, $\Rinc(x,y)$ if $x,y \in \hat{S}$.
This finishes the proof of the lemma.
\end{proof}

Summarizing, we obtain the following statement, which says that for every distinct $x,y \in \hat{S}$,
the existence or non-existence of an edge $xy$ can be determined by examining the neighbourhoods of $x$ and $y$ in $A\cup B$.
\begin{theorem}\label{thm:nice:nei}
For every nice structured pair $(G,(A,B,S))$ there exists a set $\hat{S} \subseteq S$
with $|\hat{S}| \geq |S|/4$ such that for every $x,y \in \hat{S}$ with $N_{A \cup B}(x) \neq N_{A \cup B}(y)$
the following holds.
\begin{enumerate}
\item $xy \in E(G)$ if and only if exactly one of the following holds: $\ReqA(x,y)$, $\ReqB(x,y)$, $\Rsub(x,y)$, or $\Rsup(x,y)$.
\item $xy \notin E(G)$ if and only if $\Rinc(x,y)$.
\end{enumerate}
\end{theorem}
\begin{proof}
We obtain the set $\hat{S}$ from Lemma~\ref{lem:nice:getS}. The ``if'' part of the assertion for edges $xy$ and the ``only if'' part of the assertion for nonedges $xy$ is straightforward from Lemma~\ref{lem:nice:getS}, while
the remaining two implications follow from Property (E1).
\end{proof}

As mentioned at the beginning of this section, we now partition $\hat{S}$ according to the relation $\Req$.
Observe that due to Theorem~\ref{thm:nice:nei}, the presence or absence of an edge between two vertices $x,y \in S$ is determined by $N_{A \cup B}(x)$ and $N_{A \cup B}(y)$ unless $\Req(x,y)$.
An immediate corollary is the following.

\begin{corollary}\label{cor:nice:modules}
Let $(G,(A,B,S))$ be a nice structured pair, let $\hat{S} \subseteq S$ be the set obtained from Theorem~\ref{thm:nice:nei},
and let $S_1,S_2,\ldots,S_r$ be the equivalence classes of the relation $\Req$ restricted to $\hat{S}$.
Then every set $S_i$ is a module of $G[\hat{S}]$.
\end{corollary}

As a last step in our analysis of nice structured graphs, we investigate $P_3$'s in the quotient graph of the aforementioned
partition of $G[\hat{S}]$ into modules.

\begin{lemma}\label{lem:nice:P3}
Let $(G,(A,B,S))$ be a nice structured pair, let $\hat{S} \subseteq S$ be the set obtained from Theorem~\ref{thm:nice:nei},
and let $x,y,z \in \hat{S}$ be three distinct vertices belonging to different equivalence classes of the relation $\Req$,
such that $xy \in E(G)$, $yz \in E(G)$, and $xz \notin E(G)$.
Then one of the following holds:
\begin{itemize}
\item $\ReqA(x,y)$ and $\ReqB(y,z)$;
\item $\ReqB(x,y)$ and $\ReqA(y,z)$;
\item $\Rsub(x,y)$ and $\Rsub(z,y)$;
\item $\Rsup(x,y)$ and $\Rsup(z,y)$;
\end{itemize}
\end{lemma}
\begin{proof}
Since $xz \notin E(G)$, we have $\Rinc(x,z)$; in particular the sets $N_A(x)$ and $N_A(z)$ are incomparable with respect to inclusion.
If $\ReqA(x,y)$, then the only option from Theorem~\ref{thm:nice:nei} for the edge $yz$ that allows this property to happen is $\ReqB(y,z)$;
symmetrical claims follow if we swap the roles of $A$ and $B$ and/or the roles of $x$ and $z$.
In the remaining case, if neither $(x,y)$ nor $(y,z)$ belongs to $\ReqA \cup \ReqB$, then the only way
to ensure incomparability of $N_A(x)$ and $N_A(z)$ is to have $\Rsub(x,y)$ and $\Rsub(z,y)$ or
$\Rsup(x,y)$ and $\Rsup(z,y)$.
\end{proof}

In the next lemma we remark that Lemma~\ref{lem:nice:P3} already implies that 
the quotient graph of the partition of $G[\hat{S}]$ into equivalence classes
of the relation $\Req$ is Berge. 

\begin{lemma}\label{lem:nice:berge}
Let $(G,(A,B,S))$ be a nice structured pair and
 let $\hat{S} \subseteq S$ be the set obtained from Theorem~\ref{thm:nice:nei}.
Then the quotient graph of the partition of $G[\hat{S}]$ into equivalence
classes of the relation $\Req$ is Berge.
\end{lemma}
\begin{proof}
Assume that the set $\hat{S}$ contains a sequence $x_1,x_2,\ldots,x_h$
of vertices for some odd integer $h \geq 5$, such that $x_ix_{i+1} \in E(G)$
and $x_ix_{i+2} \notin E(G)$ for every $1 \leq i \leq h$ and for indices
behaving cyclically modulo $h$. Furthermore, assume that no two vertices
$x_i$ are in relation $\Req$.

Consider the edge $x_1x_2$, and let us consider four cases, depending on
which option of Theorem~\ref{thm:nice:nei} holds for this edge.
By symmetry between the sides $A$ and $B$, we need only consider the cases
$\ReqA(x_1,x_2)$ and $\Rsub(x_1,x_2)$.
If $\ReqA(x_1,x_2)$, then Lemma~\ref{lem:nice:P3} applied to the
$P_3$ $x_1,x_2,x_3$ implies that $\ReqB(x_2,x_3)$. Inductively,
we infer that $\ReqA(x_i,x_{i+1})$ if $i$ is odd
and $\ReqB(x_i,x_{i+1})$ if $i$ is even. However, this leads to
a contradiction as $h$ is odd.
A similar situation happens if $\Rsub(x_1,x_2)$:
we have $\Rsub(x_i,x_{i+1})$ for odd $i$ and $\Rsup(x_i,x_{i+1})$ for even $i$,
again yielding a contradiction

We infer that no such sequence $x_1,x_2,\ldots,x_h$ exists.
However, note that such a sequence is present in any odd hole
in the quotient graph in the question (take the subsequent vertices
on the hole) and is present in any odd anti-hole as well
(if the anti-hole consists of $h$ vertices $y_1,y_2,\ldots,y_h$ in this order,
take $x_i = y_{(i\lfloor h/2 \rfloor)\,\mathrm{mod}\,h}$).
We infer that the quotient graph in the question does not contain any odd hole
nor anti-hole, and is thus Berge.
\end{proof}

Let us now wrap up what our analysis of nice structured graphs implies for the proof of Theorem~\ref{thm:tech-claw-free}.
Recall that we are dealing with a structured pair $(G,(A,B,S))$ where $G$ is claw-free and $C_5$-free.
Claim~\ref{lem:cf:nice} implies that (after the filtering step) we are in fact dealing with a nice structured pair.
Theorem~\ref{thm:nice:nei} provides us with a candidate set $\hat{S}$, that we fix for the remainer of this proof.
Corollary~\ref{cor:nice:modules} implies
that the relation $\Req$ partitions $G[\hat{S}]$ into modules. Moreover, by construction, every such module $S_i$
is contained in a neighbourhood of some vertex from $A$.
It remains to analyse the quotient graph of this partition.

\subsection{The quotient graph: Excluding a diamond}

Clearly, the quotient graph of the partition of $G[\hat{S}]$ into equivalence classes of the relation $\Req$
is claw-free, since $G$ is claw-free. In the rest of this section we show that it is also diamond-free. This, together
with a characterization from~\cite{KloksKM94,MetelskyT03} showing that the class of (claw,diamond)-free graphs is exactly
the class of line graphs of triangle-free graphs, concludes the proof of Theorem~\ref{thm:tech-claw-free}.

We start by showing that the last two cases of Lemma~\ref{lem:nice:P3} cannot appear if $G$ is claw-free and $C_5$-free.
\begin{claim}\label{lem:cf:P3}
Let $x,y,z$ be as in the statement of Lemma~\ref{lem:nice:P3}.
Then either
$\ReqA(x,y)$ and $\ReqB(y,z)$ or
$\ReqB(x,y)$ and $\ReqA(y,z)$.
That is, the last two cases cannot happen.
\end{claim}
\begin{proof}
Assume the contrary; by swapping the sides $A$ and $B$ if needed, we can assume that $\Rsub(x,y)$ and $\Rsub(z,y)$.
Since $xz \notin E(G)$, the sets $N_A(x)$ and $N_A(z)$ are incomparable with respect to inclusion; let $p \in N_A(x) \setminus N_A(z)$
and $q \in N_A(z) \setminus N_A(x)$. By Claim~\ref{lem:cf:Anei}, we have $pq \in E(G)$, since $p,q \in N_A(y)$.
Let $s \in N_B(y)$ be any vertex. Observe that $\{p,x,s,z,q\}$ induce a $C_5$ in $G$, a contradiction.
\end{proof}

We conclude with an observation that without the two cases of Lemma~\ref{lem:nice:P3} excluded in Claim~\ref{lem:cf:P3},
we cannot have a diamond in the quotient graph.

\begin{claim}\label{lem:cf:diamond}
The quotient graph of the partition of $G[\hat{S}]$ into equivalence classes of the relation $\Req$
is diamond-free.
\end{claim}
\begin{proof}
Assume the contrary.
Let $x,y,s,t \in \hat{S}$ be four distinct vertices that belong to four different equivalence classes of the relation $\Req$.
Furthermore, assume that $G[\{x,y,s,t\}]$ is isomorphic to a diamond with $xy \notin E(G)$.
By swapping the sides $A$ and $B$ if needed, by Claim~\ref{lem:cf:P3} applied to the triple $x,s,y$,
we can assume that $\ReqA(x,s)$ and $\ReqB(y,s)$.

Let us now consider two cases of Claim~\ref{lem:cf:P3} applied to the triple $x,t,y$.
If $\ReqA(x,t)$ and $\ReqB(y,t)$, then we have $N_A(s) = N_A(x) = N_A(t)$ and $N_B(s) = N_B(y) = N_B(t)$, 
giving $\Req(s,t)$, a contradiction.
If $\ReqB(x,t)$ and $\ReqA(y,t)$, then we have
\begin{equation}\label{eq:cf:diamond}
N_A(s) = N_A(x), \quad N_A(t) = N_A(y), \quad N_B(s) = N_B(y), \quad N_B(t) = N_B(x).
\end{equation}
Since $xy \notin E(G)$, we have by Theorem~\ref{thm:nice:nei} that $\Rinc(x,y)$.
By~\eqref{eq:cf:diamond} this
implies that $\Rinc(s,t)$, a contradiction to the assumption $st \in E(G)$ and Theorem~\ref{thm:nice:nei}.
\end{proof}

%% file: tech-hooks.tex
\newcommand{\favA}{\pi_A}
\newcommand{\favB}{\pi_B}
\newcommand{\muA}{\mu_A}
\newcommand{\muB}{\mu_B}

In this section we prove Theorem~\ref{thm:tech-hooks} using the same proof outline as for Theorem~\ref{thm:tech-claw-free}
from the previous section. In particular, after a filtering step we will prove that the $\eps$-structured pair at hand is actually  nice (c.f.~Definition~\ref{def:nice}), 
which allows us to apply the tools developed in Section~\ref{sec:nice}.

It will be convenient for the proof to split the constant $\eps$ into three constants $\degeps$, $\bheps$, and $\acteps$ in the following way. We show that, for every $k \geq 0$, if $\degeps, \bheps, \acteps$ are small enough positive constants that satisfy $2\degeps < \acteps$ then in every $\degeps$-structured pair $(G,(A,B,S))$ there exists either 
\begin{enumerate}
\item an anti-adjacent pair $(P,Q)$ in $G$ with $P,Q \subseteq S$, $|P|,|Q| \geq \bheps |S|$; or 
\item an active $\ell$-hook $(X,R)$ in $G$ with $\ell \geq k$, $R \subseteq S$, and $|R| \geq \acteps |S|$; or
\item a subset $\hat{S} \subseteq S$ with $|\hat{S}| \geq  |S|/5$ 
and a partition $\hat{S} = S_1 \uplus S_2 \uplus \ldots \uplus S_m$, for some $m\geq 2$, 
such that
\begin{enumerate}
\item $|S_i|\leq \degeps |S|$ for every $1\leq i \leq m$; 
\item every set $S_i$ is a module of $G[\hat{S}]$; and 
\item the quotient graph of this partition of the vertex set
of $G[\hat{S}]$ is a claw-free Berge graph.
\end{enumerate}
\end{enumerate}
Instead of giving an explicit formula for $\degeps$, $\bheps$, and $\acteps$, we will state a number of inequalities that these constants should satisfy in the course of the proof. Every such inequality will be true for sufficiently small positive constants; in particular, taking $\degeps = \bheps  = \frac{1}{200(k+10)}$ and $\acteps = \frac{1}{100(k+10)}$ will suffice.

For two disjoint vertex sets $Q$ and $D$ in a graph $G$, $\reach{Q}{D}$ denotes the set of vertices $v$ of $D$ such that in the graph $G[Q \cup D]$ there is a path from some vertex in $Q$ to $v$.
Equivalently $\reach{Q}{D} = C \cap D$, where $C$ is the union of all components of $G[Q \cup D]$ that contain at least one vertex of $Q$.


Let $(G,(A,B,S))$ be an $\degeps$-structured graph for some (small) constant $\degeps > 0$.

\subsection{Filtering}

In the proof of Theorem~\ref{thm:tech-hooks} we need a stronger filtering step than the one used for Theorem~\ref{thm:tech-claw-free}:
we need not only to discard vertices of $S$ that are adjacent to the entire set $A$ or $B$, but all vertices that are adjacent
to a large fraction of $A$ or $B$. Furthermore, we need to use a non-uniform measure on $A$ and $B$, as defined below.

For every $x \in S$, we fix one neighbour $\favA(x) \in N_A(x)$ and one neighbour $\favB(x) \in N_B(x)$.
We define a probability measure $\muA$ on $A$ by $\muA(X) = |\favA^{-1}(X)|/|S|$. That is, 
the measure $\muA$ corresponds to a random experiment where we choose a vertex $x \in S$ uniformly at random, and output $\favA(x)$.
Similarly we define a probability measure $\muB$ on $B$ using the function $\favB$.

Let $S_A = \{x \in S: \muA(N_A(x)) \geq 10\degeps\}$ and similarly let $S_B = \{x \in S: \muB(N_B(x)) \geq 10\degeps\}$. 
A standard averaging argument shows the following. 
\begin{claim}\label{lem:h:filter}
$|S_A|,|S_B| \leq |S|/10$.
\end{claim}
\begin{proof}
Consider the following random experiment: independently choose $x \in S$ uniformly at random and $p \in A$ according to the measure $\muA$. 
Since every vertex in $A$ is adjacent to at most $\degeps |S|$ vertices of $S$,
the probability that $px \in E(G)$ is at most $\degeps$.
On the other hand, conditioning on $x \in S_A$, we have $px \in E(G)$ with
probability at least $10\degeps$ by the definition of $S_A$. Consequently,
the probability that $x \in S_A$ is at most $1/10$. The proof for $S_B$ is symmetric. 
\end{proof}

By Claim~\ref{lem:h:filter}, we have $|S_A \cup S_B| \leq |S|/5$.
Consequently, by considering the pair $(G \setminus (S_A \cup S_B),(A,B,S \setminus (S_A \cup S_B)))$ instead of $(G,(A,B,S))$, and by suitably adapting the constant $\degeps$, in the rest of the proof we can assume that our $\degeps$-structured pair $(G,(A,B,S))$ has the additional property that
\begin{equation}\label{eq:h:filter}
\mbox{for all } x \in S: \  \muA(N_A(x)) < \degeps  \mbox{ and } \muB(N_B(x)) < \degeps. 
\end{equation}
However, we now need to exhibit a set $\hat{S}$ of size at least $|S|/4$
(instead of $|S|/5$ in the statement of Theorem~\ref{thm:tech-hooks}).

In the remainder of the proof, 
let us assume that, for some sufficiently
small constants $\degeps$, $\bheps$, and $\acteps$, our input structured graph $(G,(A,B,S))$ does not admit
the desired anti-adjacent pair nor the desired active hook; our goal
is to prove that $(G,(A,B,S))$ is nice and use the results of Section~\ref{sec:nice} to obtain the set $\hat{S}$.
Observe that~\eqref{eq:h:filter} already
implies Property (NE1) for $(G,(A,B,S))$.

\subsection{A generic claim to find an active hook}

We will encounter several situations that allow us to find an active hook in an $\eps$-structured pair. 
We bundle the commonalities in the following claim. 

\begin{claim}[\ABsym]\label{lem:h:grow} 
Assume there exist pairwise disjoint sets $Z,Q,D \subseteq V(G)$
such that:
\begin{enumerate}[(i)]
\item $Q,D \subseteq A$;
\item $(Z, D)$ is an anti-adjacent pair; 
\item for every $q \in Q$, there exists an integer $i \geq 0$
and an $i$-hook in $G[\{q\} \cup Z]$ with $q$ being the active vertex;
\label{p:grow:3}
\item $(|Z|+k)\degeps + (k+3)\bheps + \acteps < 1$;\label{p:grow:4}
\end{enumerate}
Then $\muA(\reach{Q}{D}) \leq |Z|\degeps + \bheps$.
\end{claim}
\begin{proof}
For a contradiction, assume that 
$\muA(\reach{Q}{D}) > |Z|\degeps + \bheps$.
Our goal is to construct an active $\ell$-hook 
$(X,R)$ with $\ell \geq k$, $R \subseteq S$, and $|R| \geq \acteps |S|$.

Let $S_0$ be the vertex set of the largest component of
$G[S \setminus N_S[Z]]$, and let $M = S \setminus (S_0 \cup N_S[Z])$.
Note that $|N_S[Z]| \leq \degeps |Z| |S|$ by Property 4.~of an $\degeps$-structured pair, 
and so
we have $|S_0 \cup M| > 3\bheps|S|$ by assumption~\eqref{p:grow:4}.
Thus, Lemma~\ref{lem:large-comp} implies that $|M| \leq \bheps |S|$.
Hence, we have
\begin{equation}\label{eq:grow:1}
|S \setminus S_0| \leq (|Z|\degeps + \bheps)|S|
\end{equation}
and, by assumption~\eqref{p:grow:4},
\begin{equation}\label{eq:grow:2} 
|S_0| > (k\degeps + (k+2)\bheps + \acteps)|S|.
\end{equation}

Since $\muA(\reach{Q}{D}) > |Z|\degeps + \bheps$,
and by the definition of $\muA$, there exists
$x \in S_0$ with $\favA(x) \in \reach{Q}{D}$. In particular, there exists a path from $Q$ to $S_0$ with all internal vertices in $D$.
Let $L$ be a shortest such path; note that it is possible that $L$ consists of a single edge, but $L$ contains at least two vertices since $Q \subseteq A$ and $S_0 \subseteq S$.

Let $q$ be the endpoint of $L$ in $Q$, $y$ be the second endpoint of $L$, and $x$ be the neighbour of $y$ on $L$ (it is possible that $x=q$).
Using assumption~\eqref{p:grow:3},
we find an integer $i_0 \geq 0$ and an $i_0$-hook with vertex set $X \subseteq \{q\} \cup Z$ and active vertex $q$. 
We lengthen this hook with the path $L$: define $i := i_0 + |V(L)|-2$,
$X_i := X \cup (V(L) \cap D)$, and $R_i := S_0$. 
Observe that, since $Z$ and $D \cup S_0$ are fully anti-adjacent, and $L$ is a shortest path from $Q$ to $S_0$ via $D$, we have that $G[X_i]$ is an $i$-hook
with $x$ being the active vertex, and $N(R_i) \cap X_i = \{x\}$
Consequently, $(X_i,R_i)$ is an active $i$-hook.

If $i \geq k$, then~\eqref{eq:grow:2} ensures that $(X_i,R_i)$ is a desired
active hook, a contradiction.
Otherwise, we use the path-growing argument of~\cite{blt2015}
to turn it into an active $k$-hook, using the slack in~\eqref{eq:grow:2} in the process.
More formally, we build a sequence of active $j$-hooks $(X_j,R_j)$
for $j=i,i+1,\ldots,k$, with $X_i \subset X_{i+1} \subset \ldots \subset X_k$,
$S_0 \supseteq R_i \supset R_{i+1} \supset \ldots \supset R_k$, and additionally
maintain that
\begin{equation}\label{eq:grow:inv}
|R_j| > \left((k-j)(\degeps + \bheps) + 2\bheps + \acteps\right)|S|.
\end{equation}
Clearly,~\eqref{eq:grow:inv} holds for $j=i$ (using~\eqref{eq:grow:2} and the fact that $R_i=S_0$),
while for $j=k$, \eqref{eq:grow:inv} gives the desired lower bound on $|R_k|$
for the active $k$-hook $(X_k,R_k)$.

Assume that an active $j$-hook $(X_j,R_j)$ has been constructed for some $j < k$.
Let $v_j \in X_j$ be the active vertex of this hook.
Let $R_{j+1}$ be the vertex set of the largest component of
$G[R_j \setminus N(v_j)]$;
by~\eqref{eq:grow:inv}, we have that $|R_j \setminus N(v_j)| > 3\bheps|S|$
as $R_j \subseteq S_0$, and Lemma~\ref{lem:large-comp} asserts that
$|R_{j+1}| \geq |R_j| - \degeps|S| - \bheps|S|$, proving~\eqref{eq:grow:inv}
for $R_{j+1}$. We take $v_{j+1}$ to be any vertex of $R_j \cap N(v_j) \cap N(R_{j+1})$; such a vertex exists by the connectivity of $G[R_j]$ and the assumption
$v_j \in N(R_j)$.
Let $X_{j+1} = X_j \cup \{v_{j+1}\}$.
A direct check shows
that the choice of $v_{j+1}$, $X_{j+1}$, and $R_{j+1}$ ensures that
$G[X_{j+1}]$ is a $(j+1)$-hook with active vertex $v_{j+1}$, and $N(R_{j+1}) \cap X_{j+1} = \{v_{j+1}\}$,
finishing the description of the construction of $(X_{j+1},R_{j+1})$.
Hence, $(X_k,R_k)$ is an active $k$-hook with $R_k \subseteq S$ and $|R_k| > \acteps |S|$, a contradiction.
This concludes the proof of the claim.
\end{proof}

In the remainer of the proof we assume that the constants
$\degeps$, $\bheps$, and $\acteps$ are sufficiently small such that
\begin{equation}\label{eq:grow:cover}
2\degeps + 3(6\degeps + \bheps) < 1.
\end{equation}
In particular this means that
assumption~\eqref{p:grow:4} of Claim~\ref{lem:h:grow} is satisfied as long as $|Z| \leq 6$.
It also means that the bound in the conclusion of Claim~\ref{lem:h:grow}
is small for $|Z| \leq 6$;
specifically, we can assume that 
any two neighbourhoods $N_A(x), N_A(y)$ of vertices $x,y \in S$, together with any
three sets $\reach{Q}{D} \subseteq A$
obtained from Claim~\ref{lem:h:grow} (applied with $|Z| \leq 6$), cannot cover the entire set $A$.

\subsection{Neighbourhoods in $A \cup B$}

\subsubsection{Non-edges inside an $A$-neighbourhood}

We start with proving an analogue of Claim~\ref{lem:cf:Anei}. 
\begin{claim}[\ABsym]\label{lem:h:Anei}
For every $x \in S$ and $p,q \in N_A(x)$, if $pq \notin E(G)$ then
$N_A(p) \setminus N_A(x) = N_A(q) \setminus N_A(x)$.
\end{claim}
\begin{proof}
By contradiction, and using the symmetry between vertices $p$ and $q$,
let us assume there exists $r \in A \setminus N_A(x)$ with $pr \in E(G)$
and $qr \notin E(G)$. Let $Z = \{p,q,r,x\}$ and observe that $G[X]$
is isomorphic to $P_4$, with $x$ being one of the internal vertices. Consequently,
the assumptions of Claim~\ref{lem:h:grow} are satisfied (with the roles
of $A$ and $B$ swapped) for $Q = N_B(x)$ and $D = B \setminus N_B(x)$, and we have $\muA(\reach{Q}{D}) \leq 4\degeps + \bheps$.
However, the connectivity of $B$ implies that $\reach{Q}{D} = D$,
a contradiction to~\eqref{eq:grow:cover}.
\end{proof}

\subsection{Neighbourhoods along a nonedge in $S$}

We start by proving Property~(NE2).
\begin{claim}[\ABsym]\label{lem:h:NE2}
For every $x,y \in S$ with $x \neq y$ and $xy \notin E(G)$, if $N_B(x) \neq N_B(y)$, then 
there is no edge between $N_A(x) \cap N_A(y)$ and $A \setminus N_A(x,y)$.
\end{claim}
\begin{proof}
Let  $z \in N_B(x) \triangle N_B(y)$ be any vertex.
Observe that the assumptions of Claim~\ref{lem:h:grow} are satisfied
for $Z = \{x,y,z\}$, $Q = N_A(x) \cap N_A(y)$ and $D = A \setminus N_A(x,y)$:
for every $q \in Q$ the graph $G[\{z,x,q,y\}]$ is a $P_4$ with $q$ being one of its internal vertices.
Hence, $\muA(\reach{Q}{D}) \leq 3\degeps + \bheps$.
Let us denote $F = \reach{Q}{D}$; our goal is to prove that $F = \emptyset$.

Assume the contrary, let $p \in F$ and $q \in Q$ with $pq \in E(G)$.
Let $D' = D \setminus F$ and $Q' = N_A(D')$; note that, by the definition of
$F$, we have $Q' \subseteq N_A(x) \triangle N_A(y)$.
Furthermore, Claim~\ref{lem:h:Anei} implies that $qq' \in E(G)$ for every
$q' \in Q'$: $q'$ has a neighbour in $D'$, while $q$ does not have such a neighbour,
and both $q$ and $q'$ belong either to $N_A(x)$ or to $N_A(y)$.

Let $z_x$ be any vertex in $N_B(x)$, $z_y$ be any vertex in $N_B(y)$,
and $Z' = \{p,q,x,y,z_x,z_y\}$. We claim that the assumptions of Claim~\ref{lem:h:grow} are satisfied for $Z'$, $Q'$, and $D'$: clearly $Z'$ and $D'$ are fully anti-adjacent by construction, so it remains only to check assumption~\eqref{p:grow:3}.

To this end, consider $q' \in Q'$. By symmetry between $x$ and $y$, assume $q' \in N_A(x) \setminus N_A(y)$. If $pq' \in E(G)$, then $G[\{p,q',x,z_x\}]$ is a $0$-hook
with $q'$ being the active vertex. If $pq' \notin E(G)$, then
$G[\{q,q',p,y,z_y\}]$ is a $1$-hook with $q'$ being the active vertex. 

By Claim~\ref{lem:h:grow}, we infer that $\muA(\reach{Q'}{D'}) \leq 5\degeps + \acteps$. However, by connectivity of $B$ we have $\reach{Q'}{D'} = D'$.
This, together with $\muA(N_A(x,y)) \leq 2\degeps$ by~\eqref{eq:h:filter}
and $\muA(F) \leq 3\degeps + \bheps$ contradicts~\eqref{eq:grow:cover}.
\end{proof}

Since the structured pair $(G,(A,B,S))$ satisfies Properties~(NE1)
and~(NE2), we can use Lemma~\ref{lem:nice:inc} in the following, where we prove  
Property~(NE3). 
\begin{claim}[\ABsym]
For every $x,y \in S$ with $x \neq y$ and $xy \notin E(G)$,
if $N_A(x) \subsetneq N_A(y)$, then the sets $N_A(x)$ and $N_A(y) \setminus N_A(x)$ are fully adjacent.
\end{claim}
\begin{proof}
Since $N_A(x) \subsetneq N_A(y)$, Lemma~\ref{lem:nice:inc} implies 
that $N_B(x) \neq N_B(y)$. Consequently, Claim~\ref{lem:h:NE2}
asserts that $D := A \setminus N_A(y)$ and $N_A(x)$ are fully anti-adjacent.
That is, if we define $Q = N_A(D)$, then $Q \subseteq N_A(y) \setminus N_A(x)$.

By contradiction, assume there exists $z \in N_A(y) \setminus N_A(x)$
and $p \in N_A(x)$ with $pz \notin E(G)$. Claim~\ref{lem:h:Anei}
implies that $z \notin Q$, as $p \notin Q$ and $p,z \in N_A(y)$.
Furthermore, Claim~\ref{lem:h:Anei} also implies that $z$ is fully adjacent to $Q$.
We also know that $p$ is fully adjacent to $Q$.
We infer that the conditions of Claim~\ref{lem:h:grow} are satisfied for
$Z = \{z,p,x\}$ and the sets $Q$ and $D$: for every $q \in Q$, the graph
$G[\{z,q,p,x\}]$ is a $P_4$ with $q$ being one of its internal vertices.
Consequently, $\muA(\reach{Q}{D}) \leq 3\degeps+\bheps$, which stands
in contradiction with the connectivity of $G[A]$ and~\eqref{eq:grow:cover}.
\end{proof}

\subsection{Neighbourhoods along an edge in $S$}

In the next three claims we prove Property~(E1).

\begin{claim}[\ABsym]\label{lem:h:e1}
For every $x,y \in G[S]$, if there is no edge between $N_A(x) \triangle N_A(y)$ and $A \setminus N_A(x,y)$,
then $N_A(x) = N_A(y)$.
\end{claim}
\begin{proof}
By contradiction, assume there exists $p \in N_A(x) \triangle N_A(y)$;
by symmetry, assume $p \in N_A(x) \setminus N_A(y)$.
Let $D = A \setminus N_A(x,y)$ and $Q = N_A(D) \subseteq N_A(x) \cap N_A(y)$.
Let $z$ be any vertex in $N_B(y)$, and let $Z = \{p,y,z\}$.
Observe that Claim~\ref{lem:h:Anei} implies that $p$ is fully adjacent to $Q$,
as they are both contained in $N_A(x)$ and $p$ does not have any neighbour in
$D$.
Consequently, the assumptions of Claim~\ref{lem:h:grow} are satisfied
for the sets $Z$, $Q$, and $D$: for every $q \in Q$, the graph
$G[\{p,q,y,z\}]$ is a $P_4$ with $q$ being one of the middle vertices.
Hence, $\muA(\reach{Q}{D}) \leq 3\degeps+\bheps$.
However, $\reach{Q}{D} = D$ by the connectivity of $A$, and we have
a contradiction with~\eqref{eq:grow:cover}.  
\end{proof}

\begin{claim}[\ABsym]\label{lem:h:e2}
For every $xy \in E(G[S])$, if $N_A(x) \setminus N_A(y) \neq \emptyset$ but
$\muA(\reach{N_A(x) \setminus N_A(y)}{A \setminus N_A(x,y)}) \leq 6\degeps + \bheps$,
then $N_B(x) = N_B(y)$.
\end{claim}
\begin{proof}
Let $F = \reach{N_A(x) \setminus N_A(y)}{A \setminus N_A(x,y)}$,
$D = A \setminus (N_A(x,y) \cup F)$, and $Q = N_A(D) \subseteq N_A(y)$.
Let $p$ be any vertex in $N_A(x) \setminus N_A(y)$
and let $z$ be any vertex in $N_B(y)$.

If $N_B(y) \not\subseteq N_B(x)$, then let $z_1$ be any vertex
of $N_B(y) \setminus N_B(x)$ and define $Z = \{x,y,z,p,z_1\}$.
Otherwise, unless $N_B(x) = N_B(y)$, Claim~\ref{lem:h:e1} implies
that there exists an edge $z_2z_3$ with $z_2 \in N_B(x) \setminus N_B(y)$
and $B \setminus N_B(x,y)$, and we take $Z = \{x,y,z,p,z_2,z_3\}$.

We claim that in both cases the sets $Z$, $Q$, and $D$ satisfy
the assumptions of Claim~\ref{lem:h:grow}. Clearly, $D$ and $Z$ are fully
anti-adjacent, so it remains to check only assumption~\eqref{p:grow:3}.
To this end, consider $q \in Q$. If $pq \in E(G)$, then
$G[\{p,q,y,z\}]$ is a $P_4$ with $q$ being one of the middle vertices.
Otherwise, Claim~\ref{lem:h:Anei} implies that $q \notin N_A(x)$, that is,
$q \in N_A(y) \setminus N_A(x)$. 
If the vertex $z_1$ exists, then $G[\{y,q,x,p,z_1\}]$ is a $1$-hook
with $q$ being the active vertex.
Finally, if the edge $z_2z_3$ exists, then $G[\{x,y,q,p,z_2,z_3]$ is a $2$-hook
with $q$ being the active vertex.

We infer that $\muA(\reach{Q}{D}) \leq 6\degeps + \bheps$. However,
the connectivity of $G[A]$ implies that $D = \reach{Q}{D}$. This is in contradiction
with~\eqref{eq:grow:cover} and the assumption $\muA(F) \leq 6\degeps + \bheps$.
\end{proof}

\begin{claim}\label{lem:h:e3}
For every $xy \in E(G[S])$, either $N_A(x) \setminus N_A(y) = \emptyset$ or $N_B(x) \setminus N_B(y) = \emptyset$.
\end{claim}
\begin{proof}
Assume the contrary. 
Since $N_A(x) \neq N_A(y)$, Claim~\ref{lem:h:e2} applied to the side $B$
instead of the side $A$ asserts that
$\muB(\reach{N_B(x) \setminus N_B(y)}{B \setminus N_B(x,y)}) > 6\degeps + \bheps$;
in particular, there exists an edge $z_1z_2 \in E(G)$ with $z_1 \in N_B(x) \setminus N_B(y)$ and $z_2 \in B \setminus N_B(x,y)$.

Define now $Z = \{x,y,z_1,z_2\}$, $Q = N_A(x) \setminus N_A(y)$, and
$D = A \setminus N_A(x,y)$. Observe that the assumptions of Claim~\ref{lem:h:grow}
are satisfied for these sets: for every $q \in Q$ the graph $G[\{x,q,y,z_1,z_2\}]$
is a $1$-hook with $q$ being its active vertex. 
Consequently, $\muA(\reach{Q}{D}) \leq 4\degeps + \bheps$, a contradiction
to Claim~\ref{lem:h:e2} and the assumption $N_B(x) \neq N_B(y)$.
\end{proof}

\subsection{Niceness and quotient graph}

Summing up, we have so far proven the following.
\begin{corollary}
The $\degeps$-structure $(G,(A,B,S))$ is nice.
\end{corollary}

Let us define relations $\Req$, $\Rinc$, $\Rsub$, $\Rsup$, $\ReqA$, and $\ReqB$
on $S$ as in Section~\ref{sec:nice}. We apply Theorem~\ref{thm:nice:nei},
obtaining a set $\hat{S} \subseteq S$ of size at least $|S|/4$;
by Corollary~\ref{cor:nice:modules}, the equivalence classes of $\Req$
restricted to $\hat{S}$ partition $\hat{S}$ into modules of $G[\hat{S}]$.
Furthermore, Lemma~\ref{lem:nice:berge} asserts that the quotient
graph of this partition is Berge.
Thus, to conclude the proof of Theorem~\ref{thm:tech-hooks}, it suffices
to show that the quotient graph of this partition is also claw-free.

\subsection{Excluding a claw in the quotient graph}

\begin{claim}
The quotient graph of the partition of $G[\hat{S}]$ into equivalence
classes of the relation $\Req$ is claw-free.
\end{claim}
\begin{proof}
By contradiction, assume there exists a claw $(t;x,y,z)$ in $G[\hat{S}]$
such that no pair of vertices from $\{t,x,y,z\}$ are in relation $\Req$.

We apply Lemma~\ref{lem:nice:P3} to three $P_3$s contained in the claw $(t;x,y,z)$.
Observe that if one of the first two outcomes happens for one of $P_3$s, 
say $\ReqA(x,t)$ and $\ReqB(y,t)$, then we have $\ReqA(z,t)$ by looking
at the $P_3$ on vertices $y,t,z$. Thus we obtain $\ReqA(x,z)$, a contradiction to the properties of $\hat{S}$ obtained from Theorem~\ref{thm:nice:nei}.
We infer that the only two possibilities are $\Rsub(x,t)$, $\Rsub(y,t)$, and $\Rsub(z,t)$, or the symmetrical option $\Rsup(x,t)$, $\Rsup(y,t)$, and $\Rsup(z,t)$.
By swapping the sides $A$ and $B$ if needed, we may assume that the first option happens, that is,
$N_A(x) \cup N_A(y) \cup N_A(z) \subseteq N_A(t)$ and $N_B(t) \subseteq N_B(x) \cap N_B(y) \cap N_B(z)$. 

Let $D = A \setminus N_A(t)$, $Q_{xy} = N_A(t) \setminus (N_A(x) \triangle N_A(y))$,
and similarly define $Q_{yz}$ and $Q_{xz}$.
Since $xy \notin E(G)$, by Theorem~\ref{thm:nice:nei} we have
$\Rinc(x,y)$ and there exists $p \in N_B(x) \setminus N_B(y)$.
Furthermore, observe that also $p \notin N_B(t)$.
We infer that the sets $Z = \{t,x,y,p\}$, $Q_{xy}$, and $D$ satisfy the assumptions of Claim~\ref{lem:h:grow}: for every $q \in N_A(x) \cap N_A(y)$
the graph $G[\{p,x,q,y\}]$ is a $P_4$ with $q$ being one of the middle vertices,
while for every $q \in N_A(t) \setminus N_A(x,y)$ the graph $G[\{t,q,x,p,y\}]$
is a $1$-hook with $q$ being its active vertex. 
Consequently, $\muA(\reach{Q_{xy}}{D}) \leq 4\degeps + \bheps$.
Symmetrically, the same conclusion holds for $Q_{yz}$ and $Q_{xz}$.

Note now that $Q_{xy} \cup Q_{yz} \cup Q_{xz} = N_A(t)$, as 
$(X \triangle Y) \cap (Y \triangle Z) \cap (Z \triangle X) = \emptyset$
for any three sets $X,Y,Z$. Consequently,
$\muA(\reach{N_A(t)}{D}) \leq 3(4\degeps + \bheps)$.
However, $\reach{N_A(t)}{D} = D$ by connectivity of $G[A]$,
and we have a contradiction with~\eqref{eq:grow:cover}.
This concludes the proof of the claim, and of Theorem~\ref{thm:tech-hooks}.
\end{proof}

%% file: bihomo-discussion.tex
%
%
%
%
In this section, we prove Theorem~\ref{er_theorem}. Both statements, $(a)$ and $(b)$, are implied by the following lemma. 

\begin{lemma}
\label{er_lemma}
Let $k>2$ be fixed and let $\mathcal{H}$ be a family of graphs such that  
\begin{itemize}
\item[$(P1)$] every $H \in \mathcal{H}$ contains a cycle of length at most $k$; or 
\item[$(P2)$] for every $H \in \mathcal{H}$,  the complement of $H$ contains a cycle of length at most $k$.
\end{itemize}
Then the class of $\mathcal{H}$-free graphs does not have the strong Erd\H{o}s-Hajnal property.
\end{lemma}

\begin{proof}
Assume first that $\cH$ is a family of graphs with Property $(P1)$, i.e.~every $H\in\cH$ 
contains a cycle of length at most $k$. 
For every $\delta >0$, we construct a graph $G_\delta$, say on $n$ vertices, 
that is $\cH$-free and that does not contain a homogeneous pair $(P,Q)$ with $|P|,|Q| \geq \delta n$. 


Fix $\delta > 0$, let $n$ be large enough, and let $G\sim G(n,p)$ 
be a random graph on $n$ vertices where every edge is present independently at random with 
probability
$$p=\frac{50}{\bheps^{2}n}.$$
Let $X^{k}$ be a random variable that counts the number of cycles of length at most $k$ in $G$, and 
for $3\leq \ell \leq k$, let $X_{\ell}$ be a random variable that counts the number of cycles of length $\ell$ in $G$. 
By linearity of expectation we have 
$$\Exp(X^{k}) = \sum_{\ell=3}^{k} \Exp(X_{\ell}) \leq \sum_{\ell=3}^{k} (pn)^\ell \leq k \left(\frac{50}{\delta^2}\right)^k=:C. $$
Therefore, by Markov's Inequality, 
\begin{equation}\label{aux1}
\Prob(X^{k} \geq 3C) \leq \frac{1}{3}.
\end{equation}
Let $Z_\bheps$ be a random variable that counts the number of homogeneous pairs 
$(P,Q)$ in $G$ with $|P|,|Q| = \left\lfloor\frac{\bheps}{2} n \right\rfloor$. 
Then 
$$\Exp(Z_\bheps) \leq 2^{n} \cdot 2^{n} \cdot (1-p)^{\bheps^{2}n^{2}/10} + 2^{n} \cdot 2^{n} \cdot p^{\bheps^{2}n^{2}/10},$$
where the first term is an upper bound on the expected number of anti-adjacent pairs $(P,Q)$ and the second term is an upper bound on the expected number of adjacent pairs $(P,Q)$. For $n$ large enough we have $p < \frac{1}{2}$, so that we can deduce 
$$\Exp(Z_{\bheps}) \leq 2^{2n+1}(1-p)^{\bheps^{2}n^{2}/10} \leq 2^{2n+1}e^{-p\bheps^{2}n^{2}/10},$$ 
where we use $1-x \leq e^{-x}$ in the last inequality. 
Therefore, by a standard first-moment argument and our choice of $p$, 
$$\Prob(Z_{\bheps} > 0) = \Prob(Z_{\bheps} \geq 1)\leq 
\Exp(Z_{\bheps}) \leq e^{(2n+1)\ln(2)-p\bheps^{2}n^{2}/10} \leq e^{-n}.$$
%
%
%
Therefore, with probability at most $\frac{1}{3}+o(1)$, $G$ satisfies $X^k\geq 3C$ or $Z_\delta >0$. 
That is, there exists a graph $G'$ that has at most $3C$ cycles of length at most $k$, and that 
has no homogeneous pair $(P,Q)$ with $|P|,|Q| = \left\lfloor\frac{\bheps}{2} n \right\rfloor$. 
Remove a vertex from every cycle of length at most $k$ to obtain a graph $G_\delta$ 
on $n'\geq n/2$ vertices with no homogeneous pair $(P,Q)$ with $|P|,|Q| \geq \bheps n'$. 
In particular, $G_\delta$ is $\mathcal{H}$-free, which proves the claim. 

Assume now that the family $\cH$ satisfies Property $(P2)$. 
Then the family $\cH^\compl:= \{H^\compl\colon H\in\cH\}$ satisfies Property $(P1)$. 
So, by the first part, for every $\delta>0$ we find a graph $G_\delta$, say on $n$ vertices, 
that is $\cH^\compl$-free and has no homogeneous pair $(P,Q)$ with $|P|, |Q|\geq \delta n$. 
But then, the collection of graphs $G_\delta^\compl$ shows that the family $\cH$ cannot 
have the strong \EH property either.  
\end{proof}

We are ready to prove Theorem~\ref{er_theorem}.

\begin{proof}[Proof of Theorem~\ref{er_theorem}]
First, observe that since $P_4^\compl = P_4$, the class of $P_4$-free graphs has the strong Erd\H{o}s-Hajnal property by the result of~\cite{blt2015}. 
To prove the implication in the other direction, notice that if $H$ is not an induced
subgraph of $P_{4}$ then either $H$ or $H^{\compl}$ contains a cycle. But then we can apply Lemma~\ref{er_lemma} to $\mathcal{H}=\{H\}$ and we are done.
Thus, we proved statement $(a)$.
Statement $(b)$ follows from Lemma~\ref{er_lemma} by taking $\mathcal{H}=\{H,H^{\compl}\}$. 
\end{proof}

%% file: conclusions.tex
We proved in this paper
that for every $k \geq 1$, the class of $\dhookfam{k}$-free
graphs has the strong Erd\H{o}s-Hajnal property.
Specifically, there exists $\eps(k)>0$ such that every $\dhookfam{k}$-free $n$-vertex graph contains a clique or an independent set of size at lest $n^{\eps(H)}$.
This result extends, e.g., the result on forbidding long paths and antipaths \cite{blt2015}.

The only trees on six vertices that cannot be obtained through the {\em substitution method} described in \cite{aps2001} are the path $P_6$ and the 2-hook, also known as the $E$-graph. 
Therefore, Conjecture \ref{conj:EHLight} is now known to be true for every tree $H$ on at most six vertices.  

The question of excluding pairs of graphs in the context of the  Erd\H{o}s-Hajnal conjecture was considered also in the directed setting (see: \cite{ pairs}). 
The directed version of the conjecture is equivalent to the undirected one and was recently heavily investigated (\cite{choromanski, choromanski2, chorojeb, upper, 
pseudo}). 
In the directed setting the analogue of the complement of the graph is the graph obtained by reversing directions of all the edges.
It would be interesting to see whether techniques presented in this paper can be applied in the directed setting to get generalisations of some of the known results.